\documentclass[a4paper,12pt]{amsart}

\usepackage{array,calc,youngtab}

\usepackage{t1enc}
\def\frak{\mathfrak}

\def\Cal{\mathcal}

\let\i=\iota

\def\bu{\bullet}

\def\wt#1{\widetilde{#1}}
\def\type#1#2{
  \big[\! \begin{smallmatrix} #1\\#2
  \end{smallmatrix} \!\bigr] }

\def\supertiny{\fontsize{1pt}{1pt}\selectfont}
\newlength\celldim
\newlength\fontheight
\newlength\extraheight
\def\squaretab{
  \setlength\celldim{1.7em}
  \settoheight\fontheight{.}
  \setlength\extraheight{\celldim - \fontheight}
  \newcolumntype{S}
  {@{}
  >{\centering\arraybackslash}
  p{\celldim}
  <{\rule[-0.5\extraheight]{0pt}%
  {\fontheight + \extraheight}}
  @{}}
}
\def\hlower#1#2{
  \newlength\shift \settowidth{\shift}{#2}
  \hspace{-#1} #2 \hspace{-\shift} \hspace{#1}}
\def\myng#1{{\Yvcentermath1\scriptsize\yng(#1{})}}

\def\rmh{{\rm h}}

\def\trace{{\rm trace}}
\def\bbD{\mathbb{D}}

\def\bbH{\mathbb{H}} 
\def\bbR{\mathbb{R}}
\def\bbS{\mathbb{S}}

\def\E{\mathbb{E}}

\def\R{\mathbb{R}}
\def\W{\mathbb{W}}
\def\X{\mathbb{X}}
\def\Y{\mathbb{Y}}
\def\Z{\mathbb{Z}}
\def\N{\mathbb{N}}

\def\cA{\mathcal{A}}
\def\cB{\mathcal{B}}
\def\cC{\mathcal{C}}
\def\cD{\mathcal{D}}
\def\cE{\mathcal{E}}
\def\cI{\mathcal{I}}
\def\cK{\mathcal{K}}
\def\cV{\mathcal{V}}
\def\cT{\mathcal{T}}

\def\cL{\mathcal{L}}

\def\cW{\mathcal{W}}

\def\ga{\gamma}
\def\de{\delta}

\def\si{\sigma}

\def\ph{\varphi}

\def\ps{\psi}
\def\om{\omega}

\def\De{\Delta}

\def\Ph{\Phi}

\def\Up{\Upsilon}
\def\na{\nabla}

\def\form#1{\mathbf{#1}}

\newcommand{\id}{\operatorname{id}}
\newcommand{\im}{\operatorname{im}}
\newcommand{\End}{\operatorname{End}}
\newcommand{\Mat}{\operatorname{Mat}}

\newcommand{\lpl}{
  \mbox{$
  \begin{picture}(12.7,8)(-.5,-1)
  \put(2,0.2){$+$}
  \put(6.2,2.8){\oval(8,8)[l]}
  \end{picture}$}}
\renewcommand{\vec}[1]{\mathbf{#1}}

\newcommand{\wh}{\widehat}

\newcommand{\newc}{\newcommand}

\newtheorem{theorem}{Theorem}[section]
\newtheorem{lemma}[theorem]{Lemma}
\newtheorem*{lemma*}{Lemma}
\newtheorem{proposition}[theorem]{Proposition}
\newtheorem{corollary}[theorem]{Corollary}
\theoremstyle{remark}
\newtheorem{remark}[theorem]{\rm\bf Remark}
\newtheorem*{remark*}{\rm\bf Remark}
\newtheorem{definition}[theorem]{\rm\bf Definition}

\newcommand{\ce}{{\Cal E}}

\usepackage{amssymb,stmaryrd}
\usepackage{amscd}

\newcommand{\nd}{\nabla}

\newcommand{\Rho}{P}

\newcommand{\nn}[1]{(\ref{#1})}

\newcommand{\bg}{\mbox{\boldmath{$ g$}}}

\newcommand{\J}{J}

\newcommand{\ct}{{\Cal T}}

\newcommand{\bC}{\mbox{\boldmath{$C$}}}

\newc{\aR}{\mbox{\boldmath{$ R$}}}
\newc{\aS}{\mbox{\boldmath{$ S$}}}
\newc{\aDeR}{\mbox{\boldmath{$ U$}}_B{}^P{}_C{}^Q}
\newc{\aDe}{\mbox{\boldmath$ \Delta$}}
\newc{\aNd}{\mbox{\boldmath$ \nabla$}}

\newc{\aK}{\mbox{\boldmath{$ K$}}}
\newc{\aL}{\mbox{\boldmath{$ L$}}}

\def\sideremark#1{\ifvmode\leavevmode\fi\vadjust{\vbox to0pt{\vss
 \hbox to 0pt{\hskip\hsize\hskip1em
 \vbox{\hsize3cm\tiny\raggedright\pretolerance10000
 \noindent #1\hfill}\hss}\vbox to8pt{\vfil}\vss}}}%
                        
\def\idx#1{{\em #1\/}}

\author{A. Rod Gover and Josef \v Silhan}
\email{gover@math.auckland.ac.nz} 
\title{Higher symmetries of the conformal powers of the Laplacian on conformally flat manifolds}

\begin{document}

\begin{abstract}
  On locally conformally flat manifolds we describe a construction
  which maps generalised conformal Killing tensors to differential
  operators which may act on any conformally weighted tensor bundle;
  the operators in the range have the property that they are
  symmetries of any natural conformally invariant differential
  operator between such bundles. These are used to construct all
  symmetries of the conformally invariant powers of the Laplacian
  (often called the GJMS operators) on manifolds of dimension at least
  3. In particular this yields all symmetries of the powers of the
  Laplacian $\Delta^k$, $k\in \mathbb{Z}>0$, on Euclidean space
  $\mathbb{E}^n$. The algebra formed by the symmetry operators is
  described explicitly.
\end{abstract}

\address{ARG: Department of Mathematics\\
  The University of Auckland\\
  Private Bag 92019\\
  Auckland 1142\\
  New Zealand, and \\
~Mathematical Sciences Institute, 
Australian National University, ACT 0200, Australia} \email{gover@math.auckland.ac.nz}
\address{JS: Department of Mathematics and Statistics \\
Masaryk University \\
Kotl\'a\v{r}sk\'a 2 \\
611 37 Brno \\
Czech Republic} \email{silhan@math.muni.cz}

\thanks{ARG gratefully acknowledges support from the Royal Society of
New Zealand via Marsden Grants 06-UOA-029 and 10-UOA-113. JS was supported
by the Max-Planck-Institute f\"{u}r Mathematik in Bonn and by the Grant 
agency of the Czech republic under the grant P201/12/G028. The support of 
MPIM and GACR are gratefully acknowledged.}

\maketitle

\pagestyle{myheadings}
\markboth{Gover \& \v Silhan}{Higher symmetries of the GJMS operators}

\section{Introduction}

Given a differential operator $P$, say on functions, it is natural to
consider smooth differential operators which locally preserve the
solution space of $P$. A refinement is to seek differential operators
$S$ with the property that $P\circ S=S'\circ P$, for some other
differential operator $S'$. In this case we shall say that $S$ is a
symmetry of $P$. On Euclidean $n$-space $\mathbb{E}^n$ with $n\geq 3$
the space of first order symmetries of the Laplacian $\Delta$ is
finite dimensional with commutator subalgebra isomorphic to
$\frak{so}(n+1,1)$, the Lie algebra of conformal motions of
$\mathbb{E}^n$. Second order symmetries have applications in the
problem of separation of variables for the Laplacian, see \cite{M-book} and
references therein; on $\mathbb{E}^3$ the second order symmetries were
classified by Boyer et al.\ \cite{BKM}.  

Symmetries are closely related to conformal Killing tensors and their
generalisations, see Theorem \ref{main} below. Such operators also
play a role in physics \cite{Mik32,Vas38}.  Partly motivated by these
links, Eastwood has recently given a complete algebraic description of
the symmetry algebra for the Laplacian on $\mathbb{E}^{n\geq 3}$
\cite{EaLap}.  His treatment uses conformal geometry and in particular
a treatment of the conformal Laplacian due to Hughston and Hurd
\cite{HH21} based on the classical model of the conformal $n$-sphere
as the projective image of an indefinite quadratic variety in
$\mathbb{R}^{n+2}$.  There are close links to the Fefferman-Graham
ambient metric \cite{FGast,FGnewa}, which provides a curved version of
this model, and the ideas of Maldacena's AdS/CFT correspondence
\cite{Mald,GKP,W} (as explained in \cite{EaLap}).  Eastwood's work was
extended in \cite{EaLeLap2}, via similar techniques, where the authors
found the symmetry algebra for $\Delta^2$ on $\mathbb{E}^{n\geq 3}$.

Here the first main result of the article is a simultaneous treatment of all
powers of the Laplacian on pseudo-Euclidean space $\mathbb{E}^{s,s'}$
(i.e. $\mathbb{R}^{s+s'}$ equipped with a constant signature $(s,s')$
metric) with $s+s'\geq 3$; we obtain an explicit construction of all
symmetries and a description of the algebra these generate. See
Theorem \ref{main}, and Theorem \ref{algmain}.  (In lower dimensions a
corresponding result is not to be expected as, in that case, the space
of conformal Killing vectors is infinite dimensional). As will shortly
be clear, the problem is fundamentally linked to conformal geometry.
Thus it is natural to also formulate and treat analogous questions for
the conformally invariant generalisations $P_k$ of the powers
$\Delta^k$ ($k\in \mathbb{Z}_{>0}$) on conformally flat manifolds, and
we do this; via Theorem \ref{csymc} and surrounding discussion we see
that the algebra is again described by Theorem \ref{algmain}. In
dimension 4 the operators $P_k$ were discussed in \cite{JV}.
Conformally curved versions in general dimensions ($n\geq 2k$ if even)
are due to Paneitz ($k=2$) \cite{Pan} and Graham-Jenne-Mason-Sparling
\cite{GJMS}, and have been the subject of tremendous recent interest
in both the mathematics and physics community \cite{Chang,DD,Juhl}. For
convenience we shall refer to these operators as the GJMS operators.


A main point of the current article is to develop a universal approach
to the problem of operator symmetries; the constructions and theory
here are designed to be easily adapted to study the symmetries of
other classes of differential operators.  Indeed with minor adaption
our techniques also apply to the entire class of parabolic geometries.
Firstly, rather than work on a higher dimensional ``ambient'' manifold
(which is an idea well developed only for conformal geometry and a few
other structures), we calculate directly on the $n$ dimensional space
and use tractor calculus, many tools of which apply
simultaneously to all parabolic geometries
\cite{BEG,Gosrni,CapGotrans}.  Using this machinery we construct a map
which takes solutions of certain overdetermined PDE (solutions called
generalised conformal Killing tensors) to differential operators which
have the universality property that they are symmetries for {\em any}
conformally invariant operator between irreducible bundles.  This is
Theorem \ref{cons}.  These universal symmetry operators form an
algebra under formal composition; by construction this is a quotient
of the tensor algebra $\bigotimes \frak{so}(s+1,s'+1)$.  On the other
hand for the case of GJMS operators, Theorem \ref{csymc} states that,
conversely, all symmetries arise from the operators in this algebra.
Determining the algebra of symmetries of a given order $2k$ GJMS
operator $P_k$ then proceeds in two steps.  The order $2k$ determines
the domain (density) bundle (for $P_k$ and hence) on which the
universal symmetry operators should act. From the latter we obtain an
ideal of identities satisfied by the universal symmetries; the ideal
is specific to the domain. This is the subject of Theorem
\ref{dec2can}. A further ideal is generated by symmetries that are
trivial in a sense to be made precise below, see Theorem
\ref{extra}. The result is an explicit description in Theorem
\ref{algmain} of the ideal, the quotient of $\bigotimes
\frak{so}(s+1,s'+1)$ by which yields the symmetry algebra of $P_k$.

We thank J.\ Kadourek (of Masaryk University) for discussions, and for
ultimately providing us with a proof of a conjecture of ours, which
now forms Theorem \ref{reg}. We would also like to thank the referees
for their perceptive comments which have led to useful adjustments.

\section{The main theorems}\label{maint}

\subsection{Symmetries and triviality}

Throughout we shall retrict to conformally flat pseudo-Riemannian
manifolds $(M,g)$ of dimension $n\geq 3$ and signature $(s,s')$, or
the conformal structures $(M,[g])$ that these determine. In the spirit
of Penrose's {\em abstract index notation} \cite{PR1}, we shall denote
write $\ce^a$ as an alternative notation for $TM$ and $\ce_a$ for the
dual bundle $T^*M$. Thus for example
$\ce_{ab}=\otimes^2T^*M$. According to context we may also use {\em
  concrete} indices from time to time. That is indices refering to a
frame. All manifolds, structures, functions and tensor fields will be
taken to be smooth (i.e.\ to infinite order) and all differential
operators will be linear with smooth coefficients. Since our later
treatment generalises easily, we define here the notion of symmetry in
greater generality than is strictly needed for our main results. This
also serves to indicate the general context for the devolopments.

Suppose that $P:\mathcal{V}\to \mathcal{W}$ is a smooth differential
operator between (section spaces of) irreducible bundles.  (In our
notation we shall not distinguish bundles from their smooth section spaces.) We
shall say that linear differential operators $S: \mathcal{V}\to
\mathcal{V}$ and $S': \mathcal{W}\to \mathcal{W}$ form a $(S,S')$ a
\idx{symmetry (pair)} of $P$ if the operator compositions 
$PS$ and $S'P$ satisfy
$$
PS=S'P.
$$  An example is the pair $(TP,PT)$, where $T$ is a differential
operator $T:\mathcal{W}\to \mathcal{V}$. However for obvious reasons
such symmetries shall be termed {\em trivial}.

Following the treatment of $\Delta$ and $\Delta^2$ of
\cite{EaLap,EaLeLap2}, we note that there is an algebraic structure on the
symmetries modulo trivial symmetries as follows.  First the symmetries
of $P$ form a vector space under the obvious operations.  Then if
$(S_1,S'_1)$ and $(S_2,S'_2)$ are symmetries then so too is the
composition $(S_1S_2,S'_1S'_2)$.  So the symmetries of $P$ form an
algbera $\tilde{\mathcal{S}}$.  Next we say that two symmetries
$(S_1,S'_1)$ and $(S_2,S'_2)$ are equivalent, $(S_1,S'_1)\sim
(S_2,S'_2)$, if and only if $(S_1-S_2,S'_1-S'_2)$ is a trivial
symmetry. It is easily verified that trivial symmetries form a
two-sided ideal in the algebra $\tilde{\mathcal{S}}$ and the
quotient by this yields an algebra $\mathcal{S}$. For the case that 
$P$ is a GJMS operator it is this algebra  that we shall study in detail. 

To simplify our discussion we shall often work with just the first
operator $S: \mathcal{V}\to \mathcal{V}$ in a symmetry pair. That is
an operator $S: \mathcal{V}\to \mathcal{V}$ shall be called a symmetry
if there exists some $S': \mathcal{W}\to \mathcal{W}$ that makes
$(S,S')$ a symmetry as above. (In fact for the main class of operators
we treat it is easily verified that $S'$ is uniquely determined by
$S$.) Note that with this language, and in the class of cases satisfying 
$\mathcal{V}= \mathcal{W}$, the composition
$PS$ is a trivial symmetry if and only if $S$ is a symmetry.

\subsection{Symmetries of $\Delta^k$ on $\mathbb{E}^{s,s'}$}
We shall write $\mathbb{E}^{s,s'}$ to mean $\mathbb{R}^n$, $n=s+s'$,
equipped with the standard flat diagonal signature $(s,s')$ metric
$g$; in the $s=n$, $s'=0$ case this is $n$-dimensional Euclidean space.
Here and throughout we shall make the restriction $n\geq 3$.  In this
setting the Levi-Civita connection $\nabla$ is flat and, with tensors
expressed in terms of the standard $\mathbb{R}^n$ coordinates $x^i$,
the action of $\nabla_i$ on these agrees with $\partial/\partial x^i$.
We shall use the metric $g_{ij}$ and its inverse $g^{ij}$ to lower and
raise indices in the usual way.  For example, and capturing also our sign
convention for the Laplacian, $\Delta=g^{ij}\nabla_i \nabla_j=
\nabla^i\nabla_i$. (We use the summation convention here and below
without further mention.)

Recall that a vector field $v$ is a conformal Killing field (or
infinitesimal conformal isometry) if $\cL_v g= \rho g$ for some
function $\rho$. Otherwise written, this equation is 
$$
\nabla^i v^j + \nabla^j v^i= \rho g^{ij},
$$
and so, for solutions, $\rho=2 {\rm div~} v /n$.
Suppose now that $\ph$ is a symmetric trace-free covariant tensor satisfying 
\begin{equation}\label{ck}
\nabla^{(i} \cdots \nabla^{l}\ph^{m \cdots n)}= g^{(ij}\rho^{k \cdots n)}, \quad \mbox{ with } \quad |\{ i,\cdots ,l\}|~\mbox{ an odd integer}
\end{equation}
for some tensor $\rho^{k \cdots n}$, and where $\phi^{(i\cdots n)}$
indicates the symmetric part of the tensor $\phi^{i\cdots n}$.
Then, following \cite{EaLap}, we shall term $\ph$ a {\em generalised conformal 
Killing tensor}.

In Sections \ref{consec} below we shall 
construct a canonical 1-1 map
\begin{equation}\label{cans}
\ph \mapsto (S_\ph,S'_\ph)
\end{equation}
which takes solutions of \nn{ck} to symmetries of $\Delta^k$, see
Definition \ref{cansym} and Theorem \ref{cons} (which, in fact, deal
with a far more general setting). Although we defer the construction of
\nn{cans}, let us already term $(S_\ph,S'_\ph)$ the {\em canonical
  symmetry} corresponding to $\ph$.  Our main classification result is
that all symmetries of $\Delta^k$ arise this way, and this is established in Theorem \ref{class}. Putting these results together, on
$\mathbb{E}^{s,s'}$ we have the following.
\begin{theorem} \label{main} Let us fix $k\in \mathbb{Z}_+$.  
For the Laplacian power $\Delta^k$ on $\mathbb{E}^{s,s'}$ we have the
following.  For each  $\ph$,
a solution of \nn{ck}, there is
canonically associated a symmetry $(S_\ph,S'_\ph)$ for
$\Delta^k$ with $S_\ph$ and $S'_\ph$ each having leading term
$$
\ph^{a_1 \ldots a_p} (\na_{a_1} \cdots \na_{a_p}) \De^r .
$$
 $p\in \mathbb{Z}_{\geq 0}$,
$r\in\{ 0,1,\cdots ,k-1\}$.

Modulo trivial symmetries, any symmetry of $\Delta^k$ is a linear
combination of such pairs $(S_\ph,S'_\ph)$, with various solutions
$\ph$ of \nn{ck} as above. 
\end{theorem}

\subsection{Conformal geometry and the GJMS operators} 
Although the question of symmetries of $\Delta^k$ is not phrased in
terms of conformal geometry, it turns out that there is a deep
connection. According to the Theorem \ref{main} above, all symmetries of
$\Delta^k$ arise from the solutions of the equations \nn{ck}. As we
shall explain, these equations are each conformally covariant, and in
fact this class of equations can only be fully understood via
consideration of their conformal properties.  First note that we may
alternatively write the equation \nn{ck} as
$$
\na_{(b_0} \cdots \na_{b_{2r}} \ph_{a_1 \ldots a_p)_0}=0
$$ where we have lowered the indices for convenience and $(\cdots)_0$
indicates the trace-free part over the enclosed indices.  For a given
(say symmetric) tensor taking the trace-free part is a conformally
invariant notion.  Then for example in the case of $r=0$ this is the
well known conformal Killing tensor operator.
In that case, if (on any pseudo-Riemannian manifold $(M,g)$ of the dimension 
$n$) we replace the metric $g$ with the conformally related
$\widehat{g}:=e^{2\Upsilon} g$, where $\Up\in C^\infty (M)$, and
replace $\ph$ with $\widehat{\ph}:=e^{2p\Upsilon}\ph$ then
$$
\na^{\widehat{g}}_{(b_0}  \widehat{\ph}_{a_1 \ldots a_p)_0}=
e^{2p\Up}\na_{(b_0} \ph_{a_1 \ldots a_p)_0}.
$$ 

One may think of $\ph$ here as representing a {\em density} valued
tensor of weight $2p$. Recall that on a smooth manifold the density
bundles $\ce[w]$ are the bundles associated to the frame bundle by
1-dimensional (real) representations arising as the roots (or powers)
of the square of the determinant representation. These representations
and the associated bundles are thus naturally parametrised by {\em
weights} $w$ from $\mathbb{R}$. These weights are normalised so that
$\ce[-2n]\cong (\Lambda^{n}T^*M)^2$, and with this normalisation the
weights are often called {\em conformal weights}.  Note that
$(\Lambda^{n}T^*M)^2$ is trivialised by a choice of metric and hence so
are all the line bundles $\ce[w]$.  There is a section $\tilde{\ph}$
of $\ce_{(a_1\cdots a_p)_0}[2p]=\ce_{(a_1\cdots a_p)_0}\otimes
\ce[2p]$ which, in the trivialisation of $\ce[2p]$ afforded by $g$,
has the component $\ph$, while $\tilde{\ph}$ has the component
$\widehat{\ph}= e^{2p} \ph$ with respect to the trivialisation from
$\widehat{g}$.  Since the Levi-Civita connection (for any metric $g$)
may be viewed as a connection on the principal frame bundle it follows
immediately that it yields a connection on density
weighted tensor bundles.  Thus dropping the tilde, for $\ph\in
\ce_{(a_1\cdots a_p)_0}[2p]$ we have $ \na^{\widehat{g}}_{(b_0}
\ph_{a_1 \ldots a_p)_0}= \na_{(b_0} \ph_{a_1 \ldots a_p)_0}. $ This
means that the operator descends to a well defined differential
operator on a conformal manifold $(M,c)$.  Here $(M,c)$ means a
manifold equipped with just an an equivalence class of conformally
related metrics: if $g,\widehat{g}\in c$ then $\widehat{g}=e^{2\Up}g$
for some $\Up\in C^\infty(M)$.

Henceforth, it will be convenient to use the notation and language of
conformal densities, for further details and conventions see
e.g.\ \cite{CapGoamb} or \cite{GoPetLap}.  In particular below we shall use the {\em
  conformal metric} $\boldsymbol{g}_{ab}$ to raise and lower
indices. On a conformal manifold this is a tautological section of
$\ce_{(ab)}[2]= \ce_{(ab)}\otimes [2]$ which gives an isomorphism
$\ce^a=\ce^a[0]\cong \ce_b[2]$.  In particular, via the conformal
metric, we shall identify $\cE_{(a_1 \ldots a_p)_0}[2p+2r]$ and
$\cE^{(a_1 \ldots a_p)_0}[2r]$.  Note also that with these conventions
the Laplacian $ \Delta$ is given by $\Delta=\bg^{ab}\nd_a\nd_b=
\nd^b\nd_b\,$ and so this carries a conformal weight of $-2$.
(That is, the conformal Laplacian lowers the conformal weight by 2.)

From the partial classification of conformally invariant operators
given in \cite{ES} (which uses heavily the algebraic results of
\cite{Boe-Coll}) one easily extracts the following result. 
\begin{proposition} \label{gck}
For each pair $(p,r)$, of non-negative integers,
there is a conformally invariant operator
\begin{eqnarray} \label{1BGG}
\begin{split}
  &\cE_{(a_1 \ldots a_p)_0}[2p+2r] \to
  \cE_{(b_0 \ldots b_{2r} a_1 \ldots a_p)_0}[2p+2r] \\
&\ph_{a_1 \ldots a_p} \mapsto 
  \na_{(b_0} \cdots \na_{b_{2r}} \ph_{a_1 \ldots a_p)_0} + lot
\end{split}
\end{eqnarray}
where ``$lot$'' denotes lower order terms. 
\end{proposition}
\noindent In fact there is a larger class of similar operators, but we
shall not need the even order analogues of the operators above for our
current discussion.  An algorithm for generating explicit formulae
for these operators is given in \cite{GoQJM} (in dimension four but
same formulae hold in all dimensions \cite{Gothesis}, see also
\cite{CSS3,CaldSou}). The lower order terms are given by Ricci
curvature and its derivatives; in particular on $\mathbb{E}^{s,s'}$ we
recover the operator of \nn{ck}. On any manifold we shall term $\ph$
in the kernel of \nn{1BGG} a {\em (generalised) conformal (Killing)
  tensor}.  (The terminology generalised conformal Killing tensor was
introduced in \cite{EaLeLap2} for solutions of \nn{1BGG} in the case
$p =3 $. We use the same terminology for solutions of \nn{1BGG} in
the general case.)

By construction the GJMS operator $P_k$ is conformally invariant
\cite{GJMS}.  This means that it is a natural operator on
pseudo-Riemannian manifolds $M$ that descends to a well defined
differential operator on conformal structures,
$$ P_k:\ce[k-\frac{n}2] \to \ce[-k-\frac{n}{2}].
$$  
Recall that we say $(M,g)$ is {\em locally
  conformally flat}, if locally there is a metric $\widehat{g}$,
conformally related to $g$, so that on this neighbourhood
$(M,\widehat{g})$ is isometric to $\E^{s,s'}$. If $(M,g)$ is locally
conformally flat then in all dimensions $n\geq 3$ the operators $P_k$
exist for every $k \geq 1$.

\begin{definition} \label{consym}
Let us fix a conformal manifold $(M,c)$. Suppose that $(S,S')$ is a
pair of differential operators 
$$ 
S: \ce[k-\frac{n}2] \to \ce[k-\frac{n}2], \quad\mbox{~and~}\quad S': \ce[-k-\frac{n}2]
\to \ce[-k-\frac{n}2]
$$ on the given conformal manifold $(M,c)$. If locally
(i.e. in contractable neighbourhoods) on $(M,c)$ we have agreement of
the compositions as follows
$$
P_k S =S' P_k,
$$ as operators on $\ce[k-\frac{n}{2}]$, then we shall say that
$(S,S')$ is a conformal symmetry (pair) of $P_k$ on $(M,c)$. 
\end{definition}
\noindent Note that the definition does not require/impose naturality
properties of the pair $(S,S')$. They are simply required to be well defined
differential operators on the given $(M,c)$.

For a given conformal manifold, and suitable natural number $k$, we
may ask for some description of all conformal symmetries of
$P_k$. From Theorem \ref{main} we have the following Theorem. Here
and below we use $\ce^{(p)_0}_r$ as shorthand for the bundle
$\cE^{(a_1 \ldots a_p)_0}[2r]$ (and its section space).  We will often
write $\ph^p_r$ to denote some section of this bundle.
\begin{theorem}\label{csymc} 
Let $(M,c)$ be a (locally) conformally flat 
 manifold of signature $(s,s')$. For each non-zero $\ph\in
 \ce^{(p)_0}_r$, $p\in \mathbb{Z}_{\geq 0}$, $r\in\{ 0,1,\cdots ,k-1\}$, a
 solution of \nn{1BGG}, there is canonically associated a non-trivial
 conformal symmetry $(S_\ph,S'_\ph)$ for $P_k$, with $S_\ph$ and
 $S'_\ph$ each having leading term
$$
\ph_r^{a_1 \ldots a_p} (\na_{a_1} \cdots \na_{a_p}) \De^r .
$$

Modulo trivial symmetries, locally any conformal symmetry of $P_k$ is a linear
combination of such pairs $(S_\ph,S'_\ph)$, for various solutions
$\ph$ of \nn{1BGG}, with $p$ and $r$ in the range assumed here.
\end{theorem}
\noindent

 The question of conformal symmetries is not a priori the same
 question as that addressed in Theorem \ref{main}. However using that
 $S,S'$ and $P_k$ are well defined on $(M,c)$, we may use any metric
 $g\in c$ to calculate. This is a choice similar to choosing
 coordinates in order to calculate; indeed $g$ gives a trivialisation
 of the density bundles. Now, by working locally and choosing a flat
 metric, the result here follows immediately from Theorem \ref{main},
 since by the definition of the canonical symmetries in Definition
 \ref{cansym} and Theorem \ref{cons}, they are well defined on locally
 conformally flat conformal manifolds.

\subsection{Algebraic structure} \label{alg2}
Let us denote by $\cA_k$ the algebra of symmetries of $\De^k$ on
$\E^{s,s'}$ modulo trivial symmetries. As usual we write $n=s+s'$.
It follows from the theorem
\ref{main} we have the vector space isomorphism
\begin{equation} \label{Avec}
\cA_k \cong \bigoplus_{j=0}^\infty \bigoplus_{i=0}^{k-1} \cK_i^j
\end{equation}
where $\cK_i^j \subseteq \cE_i^{(j)_0}$ is the space of solutions
of \nn{1BGG} with $r=j$ and $p=i$.

Now we turn to the algebra structure of $\cA_k$. It is well known 
\cite{Lepowsky,CSSann}, 
and given explicitly by \nn{split} below, that
 the (finite dimensional) spaces $\cK_i^j$ are isomorphic to 
irreducible $\frak{g} := \frak{so}_{s+1,s'+1}$--modules
\begin{equation} \label{K}
\cK_i^j \cong 
\raisebox{-13pt}{$
\overbrace{\begin{picture}(60,30)
\put(0,5){\line(1,0){60}}
\put(0,15){\line(1,0){60}}
\put(0,25){\line(1,0){60}}
\put(0,5){\line(0,1){20}}
\put(10,5){\line(0,1){20}}
\put(20,5){\line(0,1){20}}
\put(50,5){\line(0,1){20}}
\put(35,10){\makebox(0,0){$\cdots$}}
\put(35,20){\makebox(0,0){$\cdots$}}
\end{picture}}^{j}
\!
\overbrace{\begin{picture}(60,30)
\put(0,15){\line(1,0){60}}
\put(0,25){\line(1,0){60}}
\put(0,5){\line(0,1){20}}
\put(10,15){\line(0,1){10}}
\put(20,15){\line(0,1){10}}
\put(50,15){\line(0,1){10}}
\put(60,15){\line(0,1){10}}
\put(35,20){\makebox(0,0){$\cdots$}}
\put(62,17){\makebox(0,0)[l]{${}_0$}}
\end{picture}}^{2i}$
}
\end{equation}
in the notation of Young diagrams. (Using the highest weights,
expressed as a vector of coefficients over the Dynkin diagram as in
\cite{BastEast}, $\cK_i^j$ corresponds to the coefficient $2i$ over
the first node, the coefficient $j$ over the second one and with
remaining coefficients equal to zero. At least this applies in
dimensions at least 5, but there is an obvious adjustment in lower
dimensions.)

We follow \cite{EaLeLap2} in the discussion of the algebraic 
structure 
of $\cA_k$. Decomposing the tensor product of two copies of 
$\frak{g} = \myng{1,1}$ we obtain 
\begin{equation} \label{gg}
\frak{g} \otimes \frak{g}  = 
\underbrace{\myng{2,2}_{\,0} \oplus \myng{2}_{\,0} \oplus \R \oplus 
\myng{1,1,1,1}}_{\textstyle \frak{g} \odot \frak{g}}
\oplus
\underbrace{\myng{2,1,1}_{\hlower{1.8ex}{0}}
\oplus \myng{1,1}}_{\textstyle \frak{g} \wedge \frak{g}}
\end{equation}
where $\odot$ is the symmetric tensor product. All these components
occur with multiplicity one.  We shall need notation for the
projections of $V_1 \otimes V_2 \in \frak{g} \otimes \frak{g}$ to some
of the irreducible components on the right hand side of the previous
display.  In particular, we put
\begin{equation} \label{ggcomp}
V_1 \boxtimes V_2 \in \myng{2,2}_{\,0}, \quad
V_1 \bu V_2 \in \myng{2}_{\,0}, \quad
\langle V_1, V_2 \rangle \in \R \quad \text{and} \quad
[V_1, V_2] \in \myng{1,1},
\end{equation}
and we write the same notation for the projections. 
Here the $\boxtimes$ denotes the Cartan product, $\langle,\rangle$
the Killing form on $\frak{g}$ (normalized as in 
\cite{EaLeLap2}) and $[,]$ is the Lie bracket. These 
projections are described explicitly in \nn{IIirr} below.
There is also the inclusion
$$
\boxtimes^{2k} \tiny\yng(1) = 
\raisebox{0.7ex}{$
\underbrace{
\supertiny \squaretab
\begin{tabular}{|S|SSSS|S|} 
\hline
&&\multicolumn{3}{|c|}{.\;.\;.}&  \\
\hline
\end{tabular}}_{2k}
$}
{}_{\,0}
\hookrightarrow
\underbrace{
\myng{1,1} \odot \myng{1,1} \odot \cdots \odot \myng{1,1}}_{2k} 
\subset
\underbrace{
\myng{1,1} \otimes \myng{1,1} \otimes \cdots \otimes \myng{1,1}}_{2k},
$$
see \nn{incl} for the explicit form.
That is, there is an (obviously unique) irreducible component
in $\bigodot^{2k} \frak{g}$ of the type specified on the left hand side.

With this notation, we obtain the following generalization of
\cite[Theorem 3]{EaLeLap2}:

\begin{theorem} \label{algmain}
The algebra $\cA_k$  is isomorphic to the tensor algebra $\bigotimes \frak{g}$
modulo the two sided ideal generated by
\begin{equation} \label{gen}
V_1 \otimes V_2 - V_1 \boxtimes V_2 - V_1 \bu V_2 - \frac{1}{2} [V_1,V_2]
+\frac{(n-2k)(n+2k)}{4n(n+1)(n+2)} \langle V_1,V_2 \rangle,
\quad V_1,V_2 \in \frak{g}
\end{equation}
and the image of $\boxtimes^{2k} \tiny\yng(1)$ in $\otimes^{2k} \frak{g}$.
\end{theorem}
\noindent Note that, from Theorem \ref{csymc}, $\cA_k$ is also the algebra of
local symmetries of $P_k$ on any conformally flat conformal manifold of
dimension $n$.  

\section{Conformal tractor calculus} \label{back}

We first recall the basic elements of tractor calculus following
\cite{CapGoamb,GoPetLap}.

\subsection{Tractor bundles}\label{tractorsect}

 Let $M$ be a smoo\-th manifold of
dimension $n\geq 3$ equipped with a conformal structure $(M,c)$ of
signature $(s,s')$.
Since the Levi-Civita connection is torsion-free, the (Riemannian)
curvature 
$R_{ab}{}^c{}_d$ is given by $ [\nd_a,\nd_b]v^c=R_{ab}{}^c{}_dv^d $ where 
$[\cdot,\cdot]$ indicates the commutator bracket.  The Riemannian
curvature can be decomposed into the totally trace-free Weyl curvature
$C_{abcd}$ and a remaining part described by the symmetric {\em
Schouten tensor} $\Rho_{ab}$, according to $
R_{abcd}=C_{abcd}+2\bg_{c[a}\Rho_{b]d}+2\bg_{d[b}\Rho_{a]c}, $ where
$[\cdots]$ indicates antisymmetrisation over the enclosed indices.
We shall write $J := P^a{}_a$. The {\em
Cotton tensor} is defined by
$$
A_{abc}:=2\nabla_{[b}\Rho_{c]a} .
$$

The standard tractor bundle over $(M,[g])$
 is a vector bundle of rank $n+2$ defined, for each $g\in c$,
by  $[\ce^A]_g=\ce[1]\oplus\ce_a[1]\oplus\ce[-1]$. 
If $\wh g=e^{2\Up}g$ ($\Up\in C^\infty(M)$), we identify  
 $(\alpha,\mu_a,\tau)\in[\ce^A]_g$ with
$(\wh\alpha,\wh\mu_a,\wh\tau)\in[\ce^A]_{\wh g}$
by the transformation
\begin{equation}\label{transf-tractor}
 \begin{pmatrix}
 \wh\alpha\\ \wh\mu_a\\ \wh\tau
 \end{pmatrix}=
 \begin{pmatrix}
 1 & 0& 0\\
 \Up_a&\delta_a{}^b&0\\
- \tfrac{1}{2}\Up_c\Up^c &-\Up^b& 1
 \end{pmatrix} 
 \begin{pmatrix}
 \alpha\\ \mu_b\\ \tau
 \end{pmatrix} ,
\end{equation}
where $\Up_a:=\nd_a \Up$.  These
identifications are consistent upon changing to a third metric from
the conformal class, and so taking the quotient by this equivalence
relation defines the {\em standard tractor bundle} $\ct$, or $\ce^A$
in an abstract index notation, over the conformal manifold.
The bundle $\ce^A$ admits an
invariant metric $ h_{AB}$ of signature $(s+1,s'+1)$ and an invariant
connection, which we shall also denote by $\nabla_a$, preserving
$h_{AB}$.  In a conformal scale $g$, these are given by
\begin{equation}\label{basictrf}
 h_{AB}=\begin{pmatrix}
 0 & 0& 1\\
 0&\bg_{ab}&0\\
1 & 0 & 0
 \end{pmatrix}
\text{ and }
\nabla_a\begin{pmatrix}
 \alpha\\ \mu_b\\ \tau
 \end{pmatrix}
  =
\begin{pmatrix}
 \nabla_a \alpha-\mu_a \\
 \nabla_a \mu_b+ \bg_{ab} \tau +\Rho_{ab}\alpha \\
 \nabla_a \tau - \Rho_{ab}\mu^b  \end{pmatrix}. 
\end{equation}
It is readily verified that both of these are conformally well-defined,
i.e., independent of the choice of a metric $g\in [g]$.  Note that
$h_{AB}$ defines a section of $\ce_{AB}=\ce_A\otimes\ce_B$, where
$\ce_A$ is the dual bundle of $\ce^A$. Hence we may use $h_{AB}$ and
its inverse $h^{AB}$ to raise or lower indices of $\ce_A$, $\ce^A$ and
their tensor products.

In computations, it is often useful to introduce 
the `projectors' from $\ce^A$ to
the components $\ce[1]$, $\ce_a[1]$ and $\ce[-1]$ which are determined
by a choice of scale.
They are respectively denoted by $X_A\in\ce_A[1]$, 
$Z_{Aa}\in\ce_{Aa}[1]$ and $Y_A\in\ce_A[-1]$, where
 $\ce_{Aa}[w]=\ce_A\otimes\ce_a\otimes\ce[w]$, etc.
 Using the metrics $h_{AB}$ and $\bg_{ab}$ to raise indices,
we define $X^A, Z^{Aa}, Y^A$. Then we
immediately see that 
\begin{equation}\label{trmet}
Y_AX^A=1,\ \ Z_{Ab}Z^A{}_c=\bg_{bc} ,
\end{equation}
and that all other quadratic combinations that contract the tractor
index vanish. 
In \eqref{transf-tractor} note that  
$\wh{\alpha}=\alpha$ and hence $X^A$ is conformally invariant. 
Using this notation the tractor $V^A$ given by
$$[V^A]_g=\begin{pmatrix}
 \alpha\\ \mu_a\\ \tau
 \end{pmatrix}
$$
may be written
\begin{equation}\label{stsplit}
V^A=\alpha Y^A +\mu^aZ^A{}_a+\tau X^A.
\end{equation}

The curvature $ \Omega$ of the tractor connection 
is defined by 
$$
[\nd_a,\nd_b] V^C= \Omega_{ab}{}^C{}_EV^E 
$$
for $ V^C\in\ce^C$.  Using
\eqref{basictrf} and the formulae for the Riemannian curvature yields
\begin{equation}\label{tractcurv}
\Omega_{abCE}= Z_C{}^cZ_E{}^e C_{abce}-2X_{[C}Z_{E]}{}^e A_{eab}~.
\end{equation}

In the following we shall also need 2-form tractors, that is
$\Lambda^2\mathcal{T}$, or in abstract indices $\ce_{[AB]}$. To
simplify notation we shall set the rule that indices labelled
sequentially by a superscript are implicitly skewed over and then denote 
skew pairs with a bold multi-index. Here we
shall need this only for valence 2 forms. This convention does not
apply to subscripts. That is, $A^0A^1$ means $[A^0A^1] = \form{A}$ but
e.g.\ the notation $A_1A_2A_3$ does not assume any implicit projection
to a tensor part. The same convention will be used for tensor indices, i.e.\
$[a^0a^1]$ means $a^0a^1 = \form{a}$.

With $\ce^k[w]$ denoting the space of $k$-forms of weight $w$, the
structure of $\mathcal{E}_{\vec{A}}=\ce_{A^0A^1}$ is
\cite{BrGodeRham,GoSiCkf}
\begin{equation} \label{comp_series_form}
  \mathcal{E}_{\vec{A}} =
  \mathcal{E}^{1}[2] \lpl \left( \mathcal{E}^2[2] \oplus
  \mathcal{E}[0] \right) \lpl \mathcal{E}^{1}[0];
\end{equation}
this means that in a choice of scale the {\em semidirect sums} $\lpl$ may be
replaced by direct sums and otherwise they indicate the composition
series structure arising from the tensor powers of
\nn{transf-tractor}.

In 
a choice of metric $g$ from the conformal class, the
projectors (or splitting operators) $X,Y,Z$ for $\mathcal{E}_A$
determine corresponding projectors $\X,\Y,\Z,\W$ for
$\mathcal{E}_{\vec{A}}$, 
These execute the  splitting of this space into four components and are given 
as follows.
\begin{center}
\renewcommand{\arraystretch}{1.3}
\begin{tabular}{c@{\;=\;}l@{\;=\;}
l@{\ $\in$\ }l}
$\Y$ & $\Y_{A^0A^1}^{\quad a^1}$ &
 $Y_{A^0}^{}Z_{A^1}^{a^1} $ &
  $\mathcal{E}_{\vec{A}}^{a^1}[-2]$ \\
$\Z$ & $\Z_{A^1  A^2}^{\, a^1 a^2}$ &
 $Z_{A^1}^{\,a^1}Z_{A^2}^{\,a^2}$ &
  $\mathcal{E}_{\vec{A}}^{\vec{a}}[-2]$ \\
$\W$ & $\W_{A^0A^1}$ &
  $X_{A^0}^{}Y_{A^1}^{}$ &
  $\mathcal{E}_{\vec{A}}[0]$ \\
$\X$ & $\X_{A^0A^1}^{\quad a^1 }$ &
$X_{A^0}^{}Z_{A^1}^{\,a^1} $ &  
  $\mathcal{E}_{\vec{A}}^{a^1}[0]$.
\end{tabular}
\end{center}
Further they satisfy 
$\X^\form{A}_{\ a} \Y_\form{A}^{\ c} = \frac{1}{2} \de_a^c$,
$\Z^\form{A}_{\, \form{a}} \Z_\form{A}^{\, \form{c}} = \de_{a^1}^{c^1}
\de_{a^2}^{c^2}$ and 
$\W^\form{A} \W_\form{A} = -\frac{1}{2} \id$, the remaining 
contractions are zero.
The explicit formula for the tractor connection is then determined by how it acts on these  (cf.\ \cite{GoSiCkf,BrGodeRham}):
\begin{equation} \label{exform}
\begin{split}
  \na_p \Y_{A^0A^1}^{\quad a^1} &=
      P_{pa_0} \Z_{A^0A^1}^{\,a^0 a^1}
    +  P_p^{\ a^1} \W_{A^0A^1} \\
  \na_p \Z_{A^0A^1}^{\,a^0a^1} &=
    - 2 \de_p^{a^0} \Y_{A^0A^1}^{\quad a^1}
    - 2 P_p^{\ a^0}  \X_{A^0A^1}^{\quad a^1} \\
  \na_p \W_{A^0A^1} &=
    - \bg_{pa^1} \Y_{A^0A^1}^{\quad a^1}
    + P_{pa^1} \X_{A^0A^1}^{\,\ \ a^1} \\
  \na_p \X_{A^0A^1} &=
      \bg_{pa^0} \Z_{A^0A^1}^{\,a^0a^1}
    -  \de_p^{a^1} \W_{A^0A^1},
\end{split}
\end{equation}

\subsection{Key differential operators}

Given a choice of conformal scale, Thomas' {\em tractor-$D$ operator} \cite{BEG}
$D_A\colon\ce_{B \cdots E}[w] \to \cE_{AB\cdots E}[w-1]$
is defined by 
\begin{equation}\label{tractorD}
D_A V:=(n+2w-2)w Y_A V+ (n+2w-2)Z_{Aa}\nabla^a V -X_A(\Delta V+w \J) V. 
\end{equation}   
This is conformally invariant, as can be checked directly using the
formulae above (or alternatively there are conformally invariant
constructions of $D$, see e.g.\ \cite{Gosrni}).  Acting on sections of
weight $w\neq 1-n/2$  \nn{tractorD} is a {\em differential splitting
  operator} since there is a bundle homomorphism which inverts $D$. In
this case it is a multiple of $X^A:\cE_{AB\cdots E}[w-1]\to \ce_{B
  \cdots E}[w] $; $X^A D_A$ is a multiple of the identity on the
domain space. This splitting operator is particularly important on
$\ce[1]$, the densities of weight 1: for non-vanishing $\si\in
\ce[1]$, $g:= \si^{-2}\bg$ is Einstein if and only if $D_A\si$ is
parallel for the tractor connection. The point is that the tractor
connection \nn{basictrf} gives a prolonged system essentially equivalent
to the equation $\nabla_{(a}\nabla_{b)_0}\si +P_{(ab)_0}\si=0$ which
controls whether the metric $g\in c$ is Einstein \cite{BEG}.

The GJMS operators on conformally flat manifolds can easily be constructed 
using the tractor $D$-operator. It turns out
$$
(-1)^k X_{A_1} \ldots X_{A_k} P_k = D_{A_1} \ldots D_{A_k}
\quad \text{on} \quad \cE_\bullet[-n/2+k],
$$
see \cite{Gosrni} for details. Here $\bullet$, in $\cE_\bullet$, denotes 
any system of tractor indices (or ${\frak so}(h)$ tensor part thereof).

In addition to the tractor-$D$ operator $D_A$, one has also 
the conformally invariant \idx{double-D} operator $\bbD_\form{A}$
and its ``square'' $\bbD^2_{AB} = -\bbD_{(A}{}^P \bbD_{|P|B)}$ defined as
\begin{eqnarray} \label{doubleD}
\begin{split}
&\bbD_\form{A} = 2( w\W_\form{A} + \X_{\form{A}}^{\,a} \na_a):
\cE_\bullet[w] \longrightarrow \cE_\form{A} \otimes \cE_\bullet[w], 
\quad w \in \bbR, \\
&\bbD^2_{AB} = -( wh_{AB} + X_{(A} D_{B)} ):
\cE_\bullet[w] \longrightarrow \cE_{(AB)} \otimes \cE_\bullet[w], 
\quad w \in \bbR.
\end{split}
\end{eqnarray}
The operator $\bbD_\form{A}$ (but with the opposite sign) was
originally defined in \cite{Goadv}.  Note that, $2 X_{[A^0}D_{A^1]} =
(n+2w-2) \bbD_\form{A}$ on $\cE_\bullet[w]$.  We shall also need the
commutation relation on $\cE_\bullet[w]$
\begin{equation} \label{[D,X]} 
  [D_A,X_B] = -2\bbD_{AB} + (n+2w)h_{AB} 
\end{equation}
from \cite{Gosrni}; alternatively this may be viewed as defining $\bbD$
as (one half of) the skew part of the left hand side. 

Finally some points of notation: In the following we shall sometimes
write $\na^q$ to denote the composition of $q$ applications of
$\na$. By context it will be clear that $q$ is not to be interpreted
as an abstract index.  Next if $\cV$ is a tensor bundle, or a tensor
product of the standard tractor bundle then for $F \in \cV$ we shall
write $F|_\boxtimes$ to denote the projection of the section $F$ to
the Cartan component (with respect to the ${\frak co}(g)$ structure,
or ${\frak so}(h)$ tensor structure, respectively) of the bundle
$\cV$.    
For example on
$\mathbb{E}^{s,s'}$ equipped with the standard flat diagonal signature
$(s,s')$ metric the equation \nn{1BGG} may be expressed as
$[\na^{2r+1} \ph ]|_\boxtimes =0$.

\section{The double-D and conformally invariant operators}
\label{Ds}

We work on $(M,[g])$, assumed to be locally conformally flat.  We
outline a rather general picture here. The theorem below provides a
general technique for the construction of symmetries of \idx{any}
conformally invariant operator that acts between irreducible
natural bundles. Moreover, since the tools used are general in nature, this
result indicates how to deal with symmetries of invariant operators on
a bigger class of structures, the so-called parabolic geometries
\cite{CSbook}. This will be taken up elsewhere.

\vspace{1ex}

Consider a conformally invariant differential operator $P: \cV \to
\cW$ between irreducible (or completely reducible will suffice)
natural bundles $\cV$ and $\cW$. More specifically, we restrict only
to subbundles of $(\bigotimes \cE_a) \otimes (\bigotimes \cE^b)
\otimes \cE[w]$ which we shall term tensor bundles. The case of spinor
bundles is however completely analogous.

Assume for a moment the general (i.e.\ possibly curved) conformal setting.
Following \cite{CapGotrans}, the double-D operator $\bbD_\form{A}$ can 
be extended to all irreducible bundles (see the discussion on the 
fundamental derivative below for details). 
This extension obeys the Leibniz rule, and since \nn{doubleD} describes
$\bbD_\form{A}$ on $\cE_\bullet[w]$, it remains to understand the action of 
$\bbD_\form{A}$ on $\cE_a \cong \cE^b[-2]$. In this case we obtain
\begin{equation} \label{fundDforms}
\bbD_\form{B} f_a = -2 \W_\form{B} f_a 
+ 2\Z_\form{B}^{\,\form{b}} \bg_{b^0a} f_{b^1} + 2\X_\form{B}^{\ b} \na_b f_a
\quad \text{for} \quad f_a \in \cE_a
\end{equation}
where $\form{B}$ is a multi-index, following the convention
introduced in the previous section.

Our use of $\bbD$ is linked to the following proposition. For
a tangent vector $\ph^a \in \cE^a$ we denote by $L_\ph$ the Lie derivative
on sections of natural bundles. Recall $\cE[w]$ is such a natural 
bundle, cf.\ the definition of $\cE[w]$ in Section \ref{maint}, as well as 
$\cE_a$ and $\cE^b$.

\begin{proposition} \label{lie}
Let $M$ be any conformally flat manifold and assume $\ph^a \in \cE^a$
is a conformal Killing vector (i.e.\ a solution of \nn{1BGG}). Then
there is a unique parallel tractor $I_\ph^\form{A} \in \cE^\form{A}$ such that $\ph^a = 2\X_\form{A}^{\: a}
I_\ph^\form{A}$ \cite{GoSiCkf}, cf.\ \nn{Iph}. Then
$$
I_\ph^\form{A} \bbD_\form{A} = L_\ph 
\quad \text{on} \quad 
(\bigotimes \cE_b) \otimes (\bigotimes \cE^c) \otimes \cE[w]. 
$$
\end{proposition}

\begin{proof}
It is sufficient to verify the theorem on $\cE[w]$ and $\cE^a$ since both
operators $L_\ph$ and $I_\ph^\form{A} \bbD_\form{A}$
obey the Leibniz rule and $\cE_b \cong \cE^a[-2]$. 
Using  \nn{doubleD} and \nn{Iph} we have 
$I_\ph^\form{A} \bbD_\form{A} = \ph^a \na_a - \frac{2}{n} (\na_a \ph^a)$ on 
$\cE[2]$. Thus using \nn{fundDforms} (and \nn{Iph} below) we obtain
\begin{align*}
I_\ph^\form{B} \bbD_\form{B} f_a &= \ph^b \na_b f_a 
+ (\na_{[a} \ph_{b]}) f^b - \frac{1}{n} (\na_b \ph^b) f_a \\
&=  \ph^b \na_b f_a - f^b \na_b \ph_a + 
f^b \big[ \frac{1}{2} (\na_b \ph_a + \na_a \ph_b) 
- \frac{1}{n} \bg_{ab} \na^c \ph_c \bigr] 
\end{align*}
on $f_a \in \cE_a[2] \cong \cE^b$.
The square bracket in the display is the conformal Killing operator, and thus
vanishes.
The equality of $L_\ph$ and $I_\ph^\form{A} \bbD_\form{A}$ on $\cE[w]$
is even simpler, and hence the general case follows.
\end{proof}
Note it  obvious from the proof that the proposition does not hold without the 
assumption that $\ph^a \in \cE^a$ is a conformal Killing vector.

The conformal invariance of the operator $P: \cV \to \cW$ (between
completely reducible, bundles $\cV$ and $\cW$) is
given by the property $L_\ph P = P L_\ph$ for every conformal Killing
field $\ph^a \in \cE^a$. That is, every conformal Killing vector
$\ph^a$ provides a symmetry of the operator $P$.

As is well known, conformal invariance can equivalently be verified
from a formula for the operator $P$. In particular for each
conformally invariant operator, and a choice of metric from the
conformal class, there is a formula in terms of the Levi-Civita
connection $\na$, its curvature, and various algebraic projections
which express the operator as a natural (pseudo-)Riemmanian
differential operator. The hallmark of conformal invariance is then
that this operator is unchanged if we use the same formula when
starting with a different metric form the conformal class.  Now, given
such a formula for $P: \cV \to \cW$, we have also the (tractor
coupled) operator $P^\na: \cV \otimes \cE_\bullet \to \cW \otimes
\cE_\bullet$ given by the same formula where $\na$ is now assumed to
be coupled Levi-Civita-tractor connection. Then $P^\na$ is also
conformally invariant.  We shall often write $P$ instead of $P^\na$ to
simplify the notation.

\begin{theorem} \label{appmain}
On a conformally flat manifold, let $P: \cV \to \cW$ be a conformally
invariant operator between completely reducible tensor natural bundles
$\cV$ and $\cW$.  Then
\begin{equation*} 
P^\na \bbD_{\form{A}_1} \cdots \bbD_{\form{A}_p} = 
\bbD_{\form{A}_1} \cdots \bbD_{\form{A}_p} P: 
\cV \to \cE_{\form{A}_1 \ldots \form{A}_p} \otimes \cW. 
\end{equation*}
\end{theorem}

\begin{proof}
It is sufficient to prove the theorem in the (globally) flat case.
First assume $p=1$ and consider a conformal Killing field $\ph^a \in
\cE^a$. Then $I_\ph$ is parallel (see e.g.\ \cite{GoMathann}, but this
follows here easily from the fact the standard tractor connection is
flat).  Then $[P^\na, I_\ph^\form{A}] =0$ and using Proposition
\ref{lie} plus the fact that $L_\ph P = P L_\ph$, from conformal the
invariance of $P$, means that $I_\ph^\form{A} [\bbD_\form{A},
  P^\na]=0$ for every conformal Killing vector $\ph^a$. The space of
conformal Killing fields on the conformally flat manifolds has the
maximal dimension, i.e.\ the dimension of the bundle
$\cE_\form{A}$. Therefore $[\bbD_\form{A},P^\na]=\bbD_\form{A} P
-P^\na \bbD_\form{A} =0$ on $\cV$.  Now it follows from the definition
of $\bbD$ that the formulae for $[\bbD_\form{A},P^\na]$ on $\cV$ and
$\cE_\bullet \otimes \cV$ formally coincide.  Since
$[\bbD_\form{A},P^\na]=0$ on $\cV$, and the tractor connection is
flat, this formula yields a zero operator on every bundle $\cE_\bullet
\otimes \cV$. Using an obvious induction, the theorem follows.
\end{proof}

\vspace{1ex}

Below we shall identify 2-form tractor fields $F_\form{A}=F_{A^1A^2}$ with
endomorphism fields of the standard tractor bundle according to the
rule $(F \sharp f)_B : = F_B{}^P f_P$ for $f_B \in \cE_B$. This also
defines the notation $\sharp$.
Moreover, we shall define $\sharp$ to be trivial on the bundles $\cE_a$ and
$\cE[w]$, and then extend this action to tensor products of $\cE_A$, $\cE_a$
and $\cE[w]$ by the Leibniz rule. Note that since $F$ is skew it
yields an (pseudo-)orthogonal action pointwise and hence preserves the
$SO(p+1,q+1)$ decompositions of tractor bundles.

Theorem \ref{appmain} above is one of the primarily tools for our
subsequent construction of symmetries. However there are some
conceptual gains in linking this to some related results and so we
complete this section with these observations.

The double-D operator discussed above reflects a more general operator
called \idx{fundamental derivative} from \cite{CapGotrans} (where it
is called the fundamental-D operator).  The specialisation of this to
conformal geometry provides, for any natural bundle $\cV$, a
conformally invariant differential operator $\cD: \cV \to \cA \otimes
\cV$, where $\cA=\Lambda^2 \cT$ is often called the {\em adjoint
  tractor bundle} (because it is modelled on $\frak{g}=\frak{so}_{s+1,s'+1}$).
Since there is a natural inclusion $\cA
\hookrightarrow \End{\cE_\bullet}$ via $\sharp$, we may form 
($(-1)$--times) 
the symmetrisation of the contracted composition, to be denoted by
$$
\cD^2 : \cV
 \to (\End{\cV}) \otimes \cV.
$$ 
In the abstract index notation we write $\cD_A{}^B$ (or
$\cD_\form{A}$, using the identification $\cA\cong \ce_{A^1A^2}$) for
the fundamental derivative and so 
$\cD^2_{AB}=- \cD^C{}_{(A}\cD_{B)C}$.

We shall use $\cD$ only on weighted tensor bundles $\cV \subseteq 
(\bigotimes \cE_a) \otimes (\bigotimes \cE^b) \otimes \cE_\bullet[w]$.
Recall the fundamental derivative obeys the Leibniz rule and actually 
$\cD_\form{A} = \bbD_\form{A}$ on irreducible bundles. (In fact, the double-D
was defined in such way in \cite{CapGotrans}.) 
To show the difference between $\bbD$ and $\cD$ and, more generally,
the analogue of \nn{doubleD} we shall need certain special tractor 
sections and their corresponding algebraic actions
on tractor bundles as follows:
\begin{eqnarray} \label{hact}
\begin{split}
  &\bbH_\form{AB} = h_{A^0B^0}h_{A^1B^1}, 
  &&\bbH_\form{A} \sharp = h_{A^0B^0}h_{A^1B^1} \,\sharp_\form{B} \\
  &\wt{\bbH}_{AD\form{BC}} = h_{(A|B^0|}h_{D)C^0}h_{B^1C^1}, 
  &&\wt{\bbH}_{AD} \sharp\sharp  = h_{(A|B^0|}h_{D)C^0}h_{B^1C^1} \, 
   \sharp_\form{B}\, \sharp_{\form{C}} 
\end{split}
\end{eqnarray}
where, as usual, we skew over the index pairs $A^0A^1$, $B^0B^1$ and
$C^0C^1$.  Here the subscript of $\sharp$ indicates which skew
symmetric component is considered as an endomorphism. That is, for
example, $(\bbH_\form{A} \sharp f)_C = h_{A^0C}f_{A^1}$ for $f_C \in
\cE_C$, and this extends to tensor powers of the tractor bundle by the
Leibniz rule.  It also indicates the order of applications of these
endomorphisms in the case of $\wt{\bbH}$.

We need $\cD$ only up to a (nonzero) scalar multiple and our choice 
will differ from \cite{luminy} by $-1$. 
Explicit formulae of $\cD$ and $\cD^2$ on weighted 
tractor bundles $\cE_\bu[w]$ are given by
\begin{eqnarray} \label{fundD}
\begin{split}
  &\cD_\form{A} = 2(w\W_\form{A} + \X_{\form{A}}^{\,a} \na_a 
   + \bbH_\form{A} \sharp) \\
  &\cD^2_{AD} = -(w h_{AD} + X_{(A}D_{D)} 
   + 4h_{(A|B^0|}\bbD_{D)B^1}\, \sharp_\form{B} 
   - 4\wt{\bbH}_{AD} \sharp \sharp )
\end{split}
\end{eqnarray}
where we skew
over $[B^0B^1]$ and $\sharp_\form{B}$ indicates the skewed symmetric
component which is considered as an endomorphism.  That is,
$\cD_\form{A} = \bbD_\form{A} + 2\bbH_\form{A}
\sharp$. 

\begin{corollary} \label{appfund}
Assume the locally conformally flat setting.
Let $P: \cV \to \cW$ be a conformally invariant operator between
irreducible weighted tensor bundles $\cV$ and $\cW$.
Then
\begin{equation*}
P^\nabla\, \cD_{\form{A}_1} \cdots \cD_{\form{A}_p} = 
\cD_{\form{A}_1} \cdots \cD_{\form{A}_p} P: 
\cV \to \cE_{\form{A}_1 \ldots \form{A}_p} \otimes \cW.
\end{equation*}
\end{corollary}

\begin{proof}
We shall use an induction. The case $p=1$ is obvious
as $\cD_\form{A} = \bbD_\form{A}$ on $\cV$ and $\cW$. 
Assume
the corollary holds for a fixed integer $p$. Then 
$$
\cD_\form{\form{A}_0} \cD_{\form{A}_1} \cdots \cD_{\form{A}_p} = 
\bbD_{\form{A}_0} \cD_{\form{A}_1} \cdots \cD_{\form{A}_p}
+ 2 \bbH_{\form{A}_0} \sharp \, \cD_{\form{A}_1} \cdots \cD_{\form{A}_p}.
$$
The operator $P$ commutes with the first term on the right hand side
using $[P,{\bbD}_{\form{A}_0}]=0$ and the inductive assumption.
Since the second term involves only $\cD_{\form{A}_1} \cdots \cD_{\form{A}_p}$
with some additional trace factors, $P$ commutes with the second term 
(using the induction) as well.
\end{proof}

\begin{lemma} \label{doublefund}
Assume the locally conformally flat setting. Then 
$[\cD_\form{A},\bbD_\form{B}] =0$ on $\cV \otimes \cE_\bullet$
for $\cV$ irreducible.
\end{lemma}

\begin{proof}
From \nn{doubleD} and \nn{fundD} we obtain
$$
[\cD_\form{A},\bbD_\form{B}] = [\cD_\form{A},\cD_\form{B}]
- 2 \cD_\form{A} \bbH_\form{B} \sharp 
+ 2 \bbH_\form{B} \sharp \cD_\form{A}
= [\cD_\form{A},\cD_\form{B}] + 4 h_{B^0A^0} \cD_{B^1A^1}.
$$
Thus contracting arbitrary sections $I^\form{A} \in \cE^\form{A}$, 
$\bar{I}^\form{B} \in \cE^\form{B}$ into the previous display
we get
$$
I^\form{A} \bar{I}^\form{B} [\cD_\form{A},\bbD_\form{B}] = 
I^\form{A} \bar{I}^\form{B} [\cD_\form{A},\cD_\form{B}]
+ 4 I^{A^1P} \bar{I}_P{}^{B^1} \cD_{A^1B^1}.
$$
We put $[I,\bar{I}]^\form{C} := 4 I^{C^0P} \bar{I}_P{}^{C^1}$.
On the one hand, $I^\form{A} \bar{I}^\form{B} [\cD_\form{A},\cD_\form{B}]$
is given \cite[Proposition, p.\ 21]{CapGotrans}. 
On the other hand, a direct computation verifies the statement on 
$\cE_\bullet[w]$, cf.\ \nn{DDdec} below.
Therefore by restricting to this case
(of $\cE_\bullet[w]$), it follows that our notation $[I,\bar{I}]$
coincides precisely with $\{I,\bar{I}\}$ used in \cite{CapGotrans}.
Thus using \cite[Proposition, p.\ 21]{CapGotrans} on 
$\cV \otimes \cE_\bullet[w]$, the lemma follows. 
\end{proof}

\begin{remark}
There is also a more conceptual proof of the previous corollary (thus
also of Theorem \ref{appmain}). Motivated by \cite[Theorem
  3.3]{CapGotrans}, we note that, at each point $x \in M$,  the section
$$
\overline{\bbD}^{(k)} \si := 
(\si,\bbD \si, \bbD^{(2)} \si = \bbD \bbD \si, \ldots, \bbD^{(k)} \si)
\in \overline{\cA}^{(k)}(\cV) \subseteq 
\cV \oplus \cE_\form{A} \otimes \cV \oplus \ldots \oplus 
\bigotimes^k \cE_\form{A} \otimes \cV,
$$ contains the data of the entire $k$-jet of $\si \in \cV$. Note
although here we assume $\cV$ is irreducible, the operator
$\overline{\bbD}^{(k)}$ is defined also on bundles of the form $\cV
\otimes \cE_\bullet$. From the general theory, the subbundle
$\overline{\cA}^{(k)}(\cV)$ (defined in the obvious way by the
display) is an induced bundle of a principle $H$--bundle where $H
\subseteq SO(s+1,s'+1)$ is a parabolic subgroup. 
It is straightforward to argue that any 
conformally invariant $k$-order operator on $\cV$ is given by
$\overline{\bbD}^{(k)}$ followed by a suitable $H$-homomorphism $\Ph$
on this subbundle. We denote this homomorphism by $\Ph_P$ in the case
of the operator $P$.

Our aim is to commute $P = \Ph_P \circ \overline{\bbD}^{(k)}$ and 
$\bbD_\form{B}$. More precisely, we put
$$
P^\na := (\id|_{\cE_\form{B}} \otimes \Ph_P) \circ \overline{\bbD}^{(k)}: 
\cE_\form{B} \otimes \cV \to \cE_\form{B} \otimes \cW.
$$
Observe the formulae for
$\overline{\bbD}^{(k)}: \cV \to  \overline{\cA}^{(k)}(\cV)$ and
$\overline{\bbD}^{(k)}: 
\cE_\form{B} \otimes \cV \to \cE_\form{B} \otimes \overline{\cA}^{(k)}(\cV)$
are formally the same. (Note the implicit $\na$ is interpreted as the coupled 
Levi-Civita-tractor connection in the latter case).
That means also the formulae for $P: \cV \to \cW$ and 
$P^\na: \cE_\form{B} \otimes \cV \to \cE_\form{B} \otimes \cW$ 
are given by the same formal expression. Hence our definition of 
$P^\na$ coincides with that given before Theorem \ref{appmain}. 

Now we are ready to show that $\cD_\form{B} P = P^\na \cD_\form{B}$
on $\cV$, i.e.\
$$
(\Ph_p \otimes \id|_{\cE_\form{B}}) \circ \overline{\bbD}^{(k)} \cD_\form{B}
= \cD_\form{B} \bigl( \Ph_p  \circ \overline{\bbD}^{(k)} \bigr):
\cV \to \cE_\form{B} \otimes \cW.
$$
Clearly $\cD_\form{B} \Ph_P = (\Ph_P \otimes \id|_{\cE_\form{B}}) \cD_\form{B}$.
Since $[\cD_\form{B},\bbD_\form{A}]=0$ from Lemma \ref{doublefund} and 
$\cD_\form{B}$ preserves
subbundles (of the space $\cD_\form{B}$ acts on),   
$(\Ph_P \otimes \id|_{\cE_\form{B}}) \bbD_{\form{A}_1} \ldots \cD_\form{B}
\ldots \bbD_{\form{A}_i}$ is conformally invariant and the previous
display follows.

Henceforth we shall write $P$ instead of $P^\na$ for simplicity.
Finally note although we have shown 
$[\cD_\form{B}, P] = 0$ only on an irreducible $\cV$, the same reasoning shows
$[\cD_\form{B}, P] = 0$ also on bundles $\cV \otimes \cE_\bullet$.
Therefore this remark offers an alternative proof of the previous corollary
(thus also of Theorem \ref{appmain}).
\end{remark}

The previous results provide an obvious way to construct symmetries
of conformally invariant operators. 
Assume the section
$$
I^{\form{A}_1^{} \ldots \form{A}_p^{}B_1^{}B_1' \ldots B_r^{}B_r'}
\in \cE^{\form{A}_1^{} \ldots \form{A}_p^{}B_1^{}B_1' \ldots B_r^{}B_r'}
$$
is parallel. Then from Theorem \ref{appmain} and Corollary \ref{appfund} the differential operators
\begin{eqnarray} \label{SbbS}
\begin{split}  
&S = I^{\form{A}_1^{} \ldots \form{A}_p^{}B_1^{}B_1' \ldots B_r^{}B_r'}
\cD_{\form{A}_1^{}} \ldots \cD_{\form{A}_p^{}} 
\cD^2_{B_1^{}B_1'} \ldots \cD^2_{B_r^{}B_r'}
\quad \text{and} \\
&\bbS = I^{\form{A}_1^{} \ldots \form{A}_p^{}B_1^{}B_1' \ldots B_r^{}B_r'}
\bbD_{\form{A}_1^{}} \ldots \bbD_{\form{A}_p^{}} 
\bbD^2_{B_1^{}B_1'} \ldots \bbD^2_{B_r^{}B_r'}
\end{split}
\end{eqnarray}
commute with $P$. That is $S$ and $\bbS$ are symmetries
of the operator $P$.

\begin{proposition} \label{fund=double}
Assume the tractor 
$I^{\form{A}_1^{} \ldots \form{A}_p^{}B_1^{}B_1' \ldots B_r^{}B_r'}$
is parallel and irreducible, $I = I|_\boxtimes$. Then $S = \bbS$
on $\cE[w]$.
\end{proposition}

\begin{proof}
Consider the parallel and irreducible tractor
$I^{\form{A}_1^{} \ldots \form{A}_p^{}B_1^{}B_1' \ldots B_r^{}B_r'}$
and the symmetry $S$ from \nn{SbbS}. Since 
$\cD_\form{A} = \bbD_\form{A} + 2 \bbH_\form{A} \sharp$, the 
difference 
\begin{equation*} 
\cD_{\form{A}_1^{}} \cD_{\form{A}_2^{}} \ldots \cD_{\form{A}_p^{}} 
\cD^2_{B_1^{}B_1'} \ldots \cD^2_{B_r^{}B_r'}-
\bbD_{\form{A}_1^{}} \cD_\form{A_2^{}} \ldots \cD_{\form{A}_p^{}} 
\cD^2_{B_1^{}B_1'} \ldots \cD^2_{B_r^{}B_r'}
\end{equation*}
lives in the trace part of 
$\cE_{\form{A}_1 \ldots \form{A}_p^{} B_1^{}B_1' \ldots B_r^{}B_r'}[w]$,
cf.\ \nn{hact}.
Therefore this difference is killed after contraction with
$I_\ph^{\form{A}_1^{} \ldots \form{A}_p^{}B_1^{}B_1' \ldots B_r^{}B_r'}$.
Repeating this argument for $\cD_{\form{A}_2}, \ldots, \cD_{\form{A}_p}$, we
obtain 
\begin{equation*} 
S = I^{\form{A}_1^{} \ldots \form{A}_p^{}B_1^{}B_1' \ldots B_r^{}B_r'}
\bbD_{\form{A}_1^{}} \ldots \bbD_{\form{A}_p^{}} 
\cD^2_{B_1^{}B_1'} \ldots \cD^2_{B_r^{}B_r'}: \cE[w] \to \cE[w].
\end{equation*}

Now we replace $\cD^2_{B_1^{}B_1'}$ in the previous display by 
$\bbD^2_{B_1^{}B_1'}$. Note 
$I^{\form{A}_1^{} \ldots \form{A}_p^{}B_1^{}B_1' \ldots B_r^{}B_r'}$
commutes with $\bbD_{\form{A}_i}$ and consider
$I^{\form{A}_1^{} \ldots \form{A}_p^{}B_1^{}B_1' \ldots B_r^{}B_r'}$
contracted with
\begin{align*} 
\cD^2_{B_1^{}B_1'} & \cD^2_{B_2^{}B_2'} \ldots \cD^2_{B_r^{}B_r'}-
\bbD^2_{B_1^{}B_1'} \cD^2_{B_2^{}B_2'} \ldots \cD^2_{B_r^{}B_r'} = \\
&=-\bigl( 4h_{(B_1|C^0|} \bbD_{B_1')C^1} \sharp_\form{C}
-4 \wt{\bbH}_{B_1B_1'} \sharp\sharp \bigr)
\cD^2_{B_2^{}B_2'} \ldots \cD^2_{B_r^{}B_r'}.
\end{align*}
where we have used \nn{fundD} and \nn{doubleD}.
The second term in the round brackets on the right hand side vanishes 
after the contraction (using trace-freeness of $I$ again) so it remains
to contract 
$I^{\form{A}_1^{} \ldots \form{A}_p^{}B_1^{}B_1' \ldots B_r^{}B_r'}$
with 
\begin{align*}
4h_{C^0(B_1} \bbD_{B_1')C^1} \sharp_\form{C}
\cD^2_{(B_2^{}B_2'} \ldots \cD^2_{B_r^{}B_r')} = 
&4(r-1) h_{(B_2B_1} \bbD_{B_1'}{}^P \cD^2_{|P|B_2'} \ldots 
\cD^2_{B_r^{}B_r')} \\
&-4(r+1) \bbD_{(B_1'B_2^{}} \cD^2_{B_1^{}B_2'} \ldots \cD^2_{B_r^{}B_r')}
\end{align*}
Here we have used the fact that the indices $B_1B_1' \ldots B_rB_r'$ 
of $I$ are symmetric (because $I$ is irreducible). 
Now the second term on the right hand side
is zero due to skew symmetry of  indices of $\bbD_{B_1'B_2^{}}$ and 
the first term vanishes after contraction with $I$ which is trace-free.
Repeating the same argument for 
$\cD^2_{B_1^{}B_2'}, \ldots, \cD^2_{B_r^{}B_r'}$, the proposition follows.
\end{proof}

Note an analogous statement to the Proposition above holds where
$\cE[w]$ is replaced by any irreducible bundle $\cV$. This may  be
proved along the same lines as in the treatment above. However since 
the details are  technical and not required here, this proof is omitted.

Finally note the operators given by \nn{SbbS} are well defined also on
bundles $\cE_\bullet[w]$. In this  setting, however, they
yield generally different operators $\cE_\bullet[w] \to
\cE_\bullet[w]$.

\section{A Construction of symmetries}\label{consec}

We are now ready to construct canonical symmetries. For a section
$\ph_r^{a_1 \ldots a_p} \in \cE^{(a_1 \ldots a_p)_0}[2r]$ we shall
define the operators $(S_\ph^{},S'_\ph)$ where $S_\ph$ and $S'_\ph$
have leading term $\ph_r^{a_1 \ldots a_p} \na_{a_1} \cdots \na_{a_p}
\De^r$. To do this we use the bijective correspondence between the
linear space of solutions of \nn{1BGG} and certain finite dimensional
$\frak{g}$--modules, cf.\ the discussion around \nn{Avec}. Explicitly,
this is given by differential prolongation in the form of a
differential splitting operator $\cE^{(a_1 \ldots a_p)_0}[2r] \to
\cE^{\form{A}_1^{} \ldots \form{A}_p^{}B_1^{}B_1' \ldots
  B_r^{}B_r'}|_\boxtimes$.  There are many ways of constructing this,
but for our current purposes the splitting operator can be
conveniently expressed using the fundamental derivative.  There is a
certain operator $\mathcal{C}$ known as the {\em curved Casimir}
\cite{CSCas} which is given by $h^{AB}\cD^2_{AB}$. (Properties of the
splitting operators coming from $\cC$ will be used in Proposition
\ref{can:prop}.)  This acts on any natural bundle and, in particular,
on weighted tractor bundles.  It can thus be iterated and we shall use
operators polynomial in $\cC$.  In particular, one gets the splitting
operator as
\begin{equation} \label{split}
  \ph_r^{a_1 \ldots a_p} \mapsto
  \Y^{\form{A}_1}{}_{a_1} \cdots \Y^{\form{A}_p}{}_{a_p} 
  Y^{B_1^{}}Y^{B_1'} \cdots Y^{B_r^{}}Y^{B_r'} \ph_r^{a_1 \ldots a_p}
  \stackrel{Q}{\longrightarrow} 
  \cE^{\form{A}_1 \ldots \form{A}_pB_1^{}B_1' \ldots B_r^{}B_r'} 
\end{equation} 
where $Q$ is an operator polynomial in $\mathcal{C}$, and hence is
polynomial in $\cD $, see \cite{CSCas,GoSiComm}.  We shall denote the
image by $I_\ph^{\form{A}_1 \ldots \form{A}_pB_1B_1' \ldots B_rB_r'}
\in \cE^{\form{A}_1 \ldots \form{A}_pB_1B_1' \ldots
  B_rB_r'}|_\boxtimes$.  The main point we need is that the tractor
$I_\ph$ is parallel if and only if $\ph$ is a solution of the operator
\nn{1BGG}.

\begin{definition} \label{cansym}
Given $\ph = \ph_r^{(a_1 \ldots a_p)_0} \in \cE^{(a_1 \ldots
a_p)_0}[2r]$, $r,p \geq 0$ 
we shall associate a differential operator $S_\ph$ as follows.
Let $I_\ph$ denote the 
tractor corresponding to $\ph$, in the sense of the discussion
surrounding \nn{split} above.
Then via \nn{SbbS},
\begin{equation} \label{can} 
S_\ph: = I_\ph^{\form{A}_1^{} \ldots \form{A}_p^{}B_1^{}B_1' \ldots B_r^{}B_r'}
\bbD_{\form{A}_1^{}} \ldots \bbD_{\form{A}_p^{}} 
\bbD^2_{B_1^{}B_1'} \ldots \bbD^2_{B_r^{}B_r'}~,
\end{equation}
is a well defined differential operator $S_\ph:\mathcal{V}\to
\mathcal{V}$, for any weighted tensor-tractor bundle $\mathcal{V}$.
\end{definition}

Assume $\ph$ is a solution of \nn{1BGG}, and so the tractor $I_\ph$ is
parallel.  It follows immediately from Theorem \ref{appmain}, and the
fact that $I_\ph$ is parallel, that $S_\ph$ is a universal symmetry
operator. That is, using also that $\ph \mapsto I_\ph$ is a splitting
operator, we have the following.
\begin{theorem} \label{cons}
On a conformally flat manifold, let $P: \cV \to \cW$ be a conformally
invariant operator between irreducible tensor bundles $\cV$ and $\cW$,
and suppose that $\ph = \ph_r^{(a_1 \ldots a_p)_0} \in \cE^{(a_1
  \ldots a_p)_0}[2r]$, $r,p \geq 0$ is a solution of \nn{1BGG}. Then
with $S_\ph: \cV \to \cV $ and $S'_\ph: \cW \to \cW$ given by
\nn{can}, the pair $(S_\ph,S'_\ph)$ is symmetry of $P$.
Assuming $P$ is the GJMS operator $P_k$ then for $\ph\neq 0$ and
$r<k$, this is a non-trivial symmetry.  
\end{theorem}

\begin{proof}
It remains to prove the last claim.  Note that acting on any density
bundle, $\ph$ is the leading symbol of the operator \nn{can}. This
follows from the construction of $S_\ph$ and is also shown by
Proposition \ref{can:prop} (which we will come to later).  Thus the
leading term does not have $\De^k$ as the right factor for $r<k$.
\end{proof}

Note that $S_\ph$ and $S'_\ph$ are not the same differential
operators. The point is that \nn{can} really defines a family of
differential operators parametrised by the space of domain bundles.

We shall henceforth only pursue the case that $P$ is a GJMS operator.
As mentioned in the proof of the theorem, $\ph$ is then the leading
symbol of the operator \nn{can}. Also note that in this case the use
of $\cD$ and $\cD^2$ rather than $\bbD$ and $\bbD^2$ (respectively) in
\nn{can} yields the same symmetry, as follows from  Proposition
\ref{fund=double}.

\noindent
{\bf Remark.}  Consider an operator $F: \cE[w] \to \cE[w]$, of 
order $\tilde{p} \geq 0$, on a smooth conformal manifold manifold
$(M,[g])$ and its symbol $\tilde{\ph}^{(a_1 \ldots a_{\tilde{p}})} \in
\cE^{(a_1 \ldots a_{\tilde{p}})}$.  Then, via the  conformal structure $[g]$
we may decompose  $\tilde{\ph}$ into
irreducibles. 
Each irreducible component $\ph$ of $\tilde{\ph}$ can be realised as
$\ph^{(a_1 \ldots a_p)_0} \in \cE^{(a_1 \ldots a_p)_0}[2r]$ where
$p=\tilde{p}-2r$.  Thus we have also the operator $S_\ph$, constructed
as above except that we here do not require $\ph$ to solve \nn{1BGG}.
 We may then
take the difference $F-S_\ph: \cE[w] \to \cE[w]$. Now the whole
procedure can be repeated for the operator $F-S_\ph$. It is clear that
after a finite number of steps we obtain the form $F = \sum_{\ph \in
  U} S_\ph$ for a (finite) index set $U \subseteq \N$. That is, given
an operator $F: \cE[w] \to \cE[w]$ on a smooth manifold $M$, any
 conformal structure on $M$ yields a decomposition of
$F$ as a sum of canonical operators $S_\ph$.  

In the other direction, the operators $S_\ph$ provide the conformally
invariant quantization introduced in \cite{DLOex}, in particular 
the special case \cite[3.1]{DLOex}. Also note the Section \ref{Ds} shows 
how to rewrite the general construction \cite{CSahs} using an affine 
connection.

\section{Classification of leading terms of symmetries}

According to the discussion following Theorem \ref{csymc}, the problem of
conformal symmetries for the GJMS operators (on locally conformally
flat manifolds) is reduced to the setting of Theorem \ref{main}. So
throughout this section we work on $\mathbb{E}^{s,s'}$ equipped with the
standard flat diagonal signature $(s,s')$ metric $g$ with $s+s'=:n\geq
3$.

\vspace{1ex}

All linear differential operators $L: \cE[w] \to \cE[w]$ may be
expressed as sums of the form
\begin{equation} \label{flatop}
  L=\sum_{p,r \geq 0} \ph_r^{a_1 \ldots a_p} (\na_{a_1} \cdots \na_{a_p}) \De^r,
  \quad \ph_r^{a_1 \ldots a_p} \in \cE^{(a_1 \ldots a_p)_0}[2r]=\ce^{(p)_0}_r; 
\end{equation}
We shall
describe the right-hand-side here as a {\em standard expression} for
$L$.  Moreover we shall typically use the notation $\ph_r^p (\odot^p
\na) \De^r$ as a shorthand for the operator $\ph_r^{a_1 \ldots a_p}
(\na_{a_1} \cdots \na_{a_p}) \De^r$ in the displayed sum (as the
details of the internal index contractions are not important for our
arguments).

We use the standard expressions as above to analyze the structure of
potential symmetries and their compositions with $\Delta^k$. In
particular we shall use the following properties/descriptions of a
given coefficient $\ph_r^p$. We shall write $o(\ph_r^p)=p+2r$ and term
this the \idx{formal order} of $\ph_r^p$ and $\ell(\ph_r^p) = p+r$
which will be termed \idx{level} of $\ph_r^p$.  (These reflect
properties of terms $\ph_r^{a_1 \ldots a_p} (\na_{a_1} \cdots
\na_{a_p}) \De^r$ and how they appear naturally in appropriate tractor
formulae. However these quantities are fully determined by the
coefficients $\ph_r^p$, so it is sufficient to consider formal order
and level of coefficients.)  We also say $\type{p}{r}$ is the
\idx{type} of $\ph_r^p$.  We shall write $o(R) = a$ and $\ell(R) = b$
if all terms of a differential operator $R: \cE[w] \to \cE[w]$ 
are of the formal order at most $a$,
respectively level at most $b$. Finally if $L$ is a symmetry of
$\Delta^k$, then we shall say $L$ is a {\em normal} symmetry (of
$\Delta^k$) if $r<k$ for all terms in the standard expression
\nn{flatop}. 
Modulo trivial symmetries, any symmetry of $\Delta^k$ may be represented
by a normal symmetry. (More generally, this holds for all operators on 
functions, cf.\ the remark following Proposition \ref{can:prop}.)

Further we shall need a suitable ordering of the terms in a standard
expression.  This will be defined via the coefficients as follows:
\begin{equation} \label{ord}
  \ph_r^p \lhd \ps_{r'}^{p'} \quad \text{iff} \quad
  \ell(\ph_r^p) < \ell(\ps_{r'}^{p'}) \ \ \text{or} \ \
  \big( \ell(\ph_r^p) = \ell(\ps_{r'}^{p'})\big) \wedge 
 \big( o(\ph_r^p) < o(\ps_{r'}^{p'}) \big). 
\end{equation}
Since the coefficient $\ph_r^p$ determines a corresponding term
in the standard expression completely, 
we shall use the ordering $\lhd$ for both
coefficients and terms of an operator \nn{flatop}.

In the following, we shall use the terminology the \idx{greatest term} (or 
coefficient) with respect to the ordering $\lhd$, the \idx{leading term}
(i.e.\ the term of the highest formal order $o$) and the
\idx{term of highest level}, which refers to the quantity $\ell$ defined 
above. We would like to emphasize that all these characteristics of terms are 
generally different.

First we shall study the canonical symmetries. Since these are
constructed using tractor operators we need a further weight type
measure as follows.  In the tractor formulae, we use strings of the
symbols $X$, $Y$, $Z$ and $\X$, $\Y$, $\Z$ and $\W$ from Section
\ref{tractorsect}.  We define the \idx{homogeneity} $\rmh(\om)$ of a
string $\om \in \{X,Y,Z,\X,\Y,\Z,\W \}$ by
\begin{gather} \label{h}
  \rmh(Y)=1, \rmh(Z)=0, \rmh(X)=-1, \rmh(\Y)=1, \rmh(\Z)=h(\W)=0, \rmh(\X)=-1 \\
  \text{and} \quad \rmh(\om_1\om_2) := \rmh(\om_1) + \rmh(\om_2) \notag
\end{gather}
where $\om_1\om_2$ means a concatenation of the strings
$\om_1$ and $\om_2$.

Now we are set to describe properties of the canonical symmetries
(and more generally operators of the form \nn{can}), as follows.
\begin{proposition} \label{can:prop}
Consider $\ph=\ph_r^p \in (\odot^p TM) \otimes \cE[2r]$ and the corresponding 
operators $S_\ph: \cE[w] \to \cE[w]$ and $S'_\ph: \cE[w'] \to \cE[w']$,
$w,w' \in \R$ given by \nn{can}. 
Then, in the standard expressions
for $S_\ph$ and $S'_{\ph}$, the following properties hold:

(i) $S_\ph$ and $S'_\ph$ have the same leading term $\ph$. 

(ii) $\ell(S'_\ph) = \ell(S_\ph) = r+p = \ell(\ph_r^p)$, that is
every term $\ps$ of $S_\ph$ or $S'_\ph$ satisfies $\ell(\ps) \leq p+r$.
Moreover, the greatest terms of $S_\ph$ and $S'_\ph$ have the coefficient
$\ph$. 

(iii) $o(S'_\ph) = o(S_\ph) = p+2r = o(\ph_r^p)$, that is
every term $\ps$ of $S_\ph$ or $S'_\ph$ satisfies $o(\ps) \leq p+2r$.
Moreover, the equality happens only for $\ps = \ph$. 

(iv) Every term $\ps$ of type $\type{\bar{p}}{\bar{r}}$ of 
$S_\ph$ or $S'_\ph$ satisfies $r \geq \bar{r}$. 
\end{proposition}

\noindent
{\bf Remark:} We shall actually use the Proposition only in the case
$[\na^{2r+1} \ph]|_\boxtimes =0$ i.e.\ when $(S_\ph^{},S'_\ph)$ is the symmetry 
pair. But note that part (iv) means, in particular, that 
any operator $L$ on functions satisfies, modulo trivial symmetries of
$\Delta^k$ that $r<k$ for all terms in the standard expression
\nn{flatop} of $L$.

\begin{proof}
First note that because $S_\ph$ and $S'_\ph$ are given by the same
operator \nn{can} acting on different density bundles, it turns out to
be sufficient to establish facts only for $S_\ph$.
From \nn{can} $S_\ph$ is 
 defined as the contraction of the parallel tractor 
$I_\ph^{\form{A}_1^{} \ldots \form{A}_p^{}B_1^{}B_1' \ldots B_r^{}B_r'}$,
corresponding to $\ph$,
with the operator 
$$ 
\widetilde{\bbD}_{\form{A}_1^{} \ldots \form{A}_p^{} B_1^{}B_1' \ldots
  B_r^{}B_r'} := \bbD_{\form{A}_1^{}} \ldots \bbD_{\form{A}_p^{}}
\bbD^2_{B_1^{}B_1'} \ldots \bbD^2_{B_r^{}B_r'}: \cE_{\bullet}[w]
\longrightarrow \cE_{\bullet\,\form{A}_1^{} \ldots \form{A}_p^{}
  B_1^{}B_1' \ldots B_r^{}B_r'}[w]. 
$$
We need some broad facts about 
the structure of the tractor formulae for $I_\ph$ and
$\widetilde{\bbD}$. When working in a metric scale and using
\nn{stsplit}, \nn{basictrf}, \nn{exform}, and \nn{doubleD} it follows that 
terms of
these are built respectively from tensor fields and tensor valued differential 
operators contracted into `projectors'
$$ 
\om \in \cB.
$$ Here $\cB$ is a set of fields taking values in the appropriate
tractor bundle tensor  product with an irreducible weighted
trace-free tensor bundle. Each element $\om\in \cB$ is an appropriate
projection (onto the irreducible part with respect to the tensor
indices) of a $p$-fold tensor product of elements from $\{ \X,\Y,\Z,\W
\} $ with a $2r$-fold tensor product of elements from $ \{ X,Y,Z \}$,
and we may take $\cB$ to be all such.
Similarly, the elements of $\cB$ can be considered as `injectors', i.e.\
a mapping going in the opposite direction.
For example, since $I_\ph$ is obtained from $\ph$ by a splitting
operator, it has the form
\begin{equation}\label{sform}
 I_\ph^{\form{A}_1^{} \ldots \form{A}_p^{}B_1^{}B_1' \ldots B_r^{}B_r'}
   = \sum_{\om \in \cB} 
   \om^{\form{A}_1^{} \ldots \form{A}_p^{}B_1^{}B_1' \ldots B_r^{}B_r'}  
   \cdot F_\om(\ph) 
\end{equation}
where, for each $\om \in \cB$, $F_{\om}(\ph)$ is the result of a
(weighted tensor valued) differential operator $F_\om$ acting on $\ph$
(a section of $(\odot^p TM) \otimes \cE[2r]$) and `$\cdot$' indicates
a contraction of tensor indices (which are suppressed); 
cf.\ \nn{Iph} below which shows $I_\ph$ for $\ph^a \in \cE^a$ explicitly.
Note also that
we sum over all strings in $\cB$ in the previous display, so many of
the $F_\om$ will be zero.  Similarly, it follows from the definition
of $\widetilde{\bbD}$ that
\begin{equation}\label{dform}
\widetilde{\bbD}_{\form{A}_1^{} \ldots \form{A}_p^{}B_1^{}B_1' \ldots B_r^{}B_r'}
= \sum_{\om \in \cB} 
\om_{\form{A}_1^{} \ldots \form{A}_p^{}B_1^{}B_1' \ldots B_r^{}B_r'} \cdot 
G_\om 
\end{equation}
where $G_\om$ is a (weighted tensor valued) differential operator
acting on densities and, again, `$\cdot$' denotes contraction of
(suppressed) tensor indices.  
See \nn{doubleD} and \nn{DD} for explicit examples.
Contracting the last two displays we
obtain the canonical symmetry $S_\ph$ as in \nn{can}.  Thus, using
\nn{trmet} and the surrounding observations, we have
$$ S_\ph = \sum_{\begin{subarray}{c} \om, \om' \in \cB,
    \\ \rmh(\om)+h(\om')=0 \end{subarray}} (F_\om(\ph)) \cdot
G_{\om'} $$ where `$\cdot$' indicates the contraction of suppressed
tensor indices.  Note pairs $(\om,\om')$ not satisfying
$\rmh(\om)+\rmh(\om')=0$ have dropped out of the sum by properties of
the tractor metric.  (Also note that the same property implies that if
the tensor indices of $F_\om$ and $G_{\om'}$ are not compatible for
complete contraction then the term $(F_\om(\ph)) \cdot G_{\om'}$ is
necessarily zero.)  

The differential order of $F_\om$ (and similarly $G_{\om'}$) is exactly 
the maximal
number of $\na$'s in the corresponding expression in the splitting
operator. (We consider formulae for splitting operators obtained using
the curved Casimir $\mathcal{C}=h^{AB}\cD^2_{AB}$ here.)
Denoting the differential order of $F_\om$ and
$G_{\om'}$ (in \nn{sform} and \nn{dform}) by, respectively, $o(F_\om)$
and $o(G_{\om'})$, we have 
$$ 
\rmh(\om) + o(F_\om) = p+2r \quad \text{and} \quad \rmh(\om') +
o(G_{\om'}) = 0, \quad \om, \om' \in \cB. 
$$ 
Here the first equality
follows from \nn{split} and the properties of splitting operators. The second
follows from the definition of $\widetilde{\bbD}$ (in
particular from the tractor expressions for $\bbD$ and $\bbD^2$ in
\nn{doubleD}), \nn{basictrf}, and \nn{exform}. 
Summing up the equalities in
the previous display we see that
\begin{equation} \label{oh}
  S_\ph = \sum_{\begin{subarray}{c} \om, \om' \in \cB, \\ 
      o(F_\om) + o(G_{\om'}) = p+2r \end{subarray}} 
  (F_\om(\ph)) \cdot G_{\om'}. 
\end{equation}
Note that all tractor indices have been eliminated, the formula
\nn{oh} for $S_\ph$ is expressed using tensor operators and
contractions only. 
Now consider a summand $(F_\om(\ph)) \cdot G_{\om'}$ of $S_\ph$ as in
\nn{oh}.  First $o(F_\om) + o(G_{\om'}) = p+2r$ implies $o(G_{\om'})
\leq p+2r$; moreover the equality can happen only if $F_\om=\id$ (up
to a non-zero scalar multiple), since \nn{split} is a differential
splitting operator. For the same reason this term does occur.
In the previous display the term with $F_\om=\id$ clearly recovers the
highest order term, i.e.\ the leading term. 
Therefore (i) follows.

Now by assumption $F_\om(\ph)$ is irreducible.
Since $S_\ph: \cE[w] \to \cE[w]$, it follows from \nn{flatop} that in
the standard expression $(F_\om(\ph)) \cdot G_{\om'} = \ga^{a_1 \ldots
  a_{\bar{p}}} \na_{a_1} \ldots \na_{a_{\bar{p}}} \De^{\bar{r}}$,
$\bar{p},\bar{r} \geq 0$
where $\ga$ is
symmetric and trace--free.  In fact, it follows from the form of
$I_\ph$ and $\wt{\bbD}$ that $F_\om(\ph) = \ga^{a_1 \ldots
  a_{\bar{p}}}$ and $G_{\om'} = \na_{a_1} \ldots \na_{a_{\bar{p}}}
\De^{\bar{r}}$.  We denote the type of $F_\om(\ph)$ by
$\type{\bar{p}}{\bar{r}}$. From this we get $o(G_{\om'}) = \bar{p} +
2\bar{r}$ and, since $F_\om$ takes $\ph$ of the type $\type{p}{r}$ to
a section of the type $\type{\bar{p}}{\bar{r}}$, we get $o(F_\om) \geq
|p-\bar{p}|$. (The point is that each application of the Levi-Civita
connection may increase or decrease the rank by 1, and this is the
only way the rank may change.) 
These properties hold for every
(irreducible) term $(F_\om(\ph)) \cdot G_{\om'}$ in \nn{oh}. Therefore
$$ p+2r = o(F_\om) + o(G_{\om'}) \geq |p-\bar{p}| + \bar{p} +
2\bar{r} $$ using \nn{oh}. We prove (ii), (iii) and (iv) separately in
cases $p \geq \bar{p}$ and $p \leq \bar{p}$.  If $p \geq \bar{p}$ then
the previous display says $p+2r \geq p + 2\bar{r}$ hence $r \geq
\bar{r}$. This implies $p+r \geq \bar{p} + \bar{r}$ and $p+2r \geq
\bar{p} + 2\bar{r}$.  If $p \leq \bar{p}$ then the previous display
means $2p+2r \geq 2\bar{p} + 2\bar{r}$ hence $p+r \geq \bar{p} +
\bar{r}$.  The latter inequality with $p \leq \bar{p}$ yields $r \geq
\bar{r}$ and so $p+2r \geq \bar{p} + 2\bar{r}$. This show (iv) and the
inequalities in (ii) and (iii). Now when equality holds in (iii) then
$p+2r = \bar{p}+ 2\bar{r}$. But then $p=\bar{p}$ from the previous
display thus also $r=\bar{r}$. This means $o(G_{\om'}) = p+2r$ and
$o(F_\om)=0$.  Hence $F_\om=id$, up to a multiple, and so if the term
is non-trivial we recover the leading term. It remains to discuss the
greatest term of $S_\ph$. But since we have already proved the
inequality in (ii), according to the ordering of \nn{ord} we need to
consider the order of terms of level $p+r$.  The
maximal order is then characterized by (iii).
\end{proof}

Note the part (iii) of the previous proposition means that the canonical 
symmetry $(S_\ph^{},S'_\ph)$, $\ph_r^p \in (\otimes^p TM) \otimes \cE[2r]$ 
is nontrivial for $P_k$, $k>r$. (The statement (iii) is actually stronger:
no term in $S_\ph$ has $\De^k$, $k>r$ as the right factor.)
\vspace{1ex}

Our strategy for classifying the  leading terms of symmetries uses
the ordering \nn{ord}. We shall start with the greatest term and study
what the symmetry condition imposes on its coefficient. We obtain
the following

\vspace{1ex}
\noindent
\idx{Claim:~} Let $\ph_i^j\in \ce^{(j)_0}_i$ is the greatest coefficient
of a symmetry $T$.  Then $[\na^{2i+1}\ph_i^j]_\boxtimes=0$.
\vspace{1ex} 

The claim forms the basis for an inductive procedure, as if
$[\na^{2i+1}\ph_i^j]_\boxtimes=0$ then the greatest term of
$T-S_{\ph_i^j}$ is strictly smaller (w.r.t.\ $\rhd$) than $\ph_i^j$, and
using Proposition \ref{can:prop}, we can replace $T$ by
$T-S_{\ph_0^p}$ and apply the previous claim again. 

The Claim is proved as Proposition \ref{biggest}, and then the detailed 
inductive procedure is  in the proof of  Theorem \ref{class}.
The proof of Proposition \ref{biggest} requires a detailed analysis of certain 
 terms. To demonstrate the technique, let us discuss an example first.
Assume that $(T,T')$ is a symmetry of $P_4 = \De^4$ of order $p$, i.e.\
$$ \De^4T=T'\De^4, \quad T = \sum_{2i+j \leq p, i<4} \ph_i^j (\odot^j
\na) \De^i $$ 
where we have displayed the standard expression of $T$.
Note we have not included terms with $i \geq
4$ as they may be eliminated by the addition of trivial symmetries of
$\De^4$. It is useful to write the terms of $T$ in a table as
follows: \\[1ex]
\noindent
\begin{tabular}{cc@{+}c@{+}c@{+}c@{+}}
order $p$: 
 & $\ph_0^{p} (\odot^p \na)$ 
 & $\ph_1^{p-2} (\odot^{p-2} \na) \De^1$
 & $\ph_2^{p-4} (\odot^{p-4} \na) \De^2$ 
 & $\ph_3^{p-6} (\odot^{p-6} \na) \De^3$ \\
order $p-1$: 
 & $\ph_0^{p-1} (\odot^{p-1} \na)$ 
 & $\ph_1^{p-3} (\odot^{p-3} \na) \De^1$
 & $\ph_2^{p-5} (\odot^{p-5} \na) \De^2$ 
 & $\ph_3^{p-7} (\odot^{p-7} \na) \De^3$ \\
order $p-2$: 
 & $\ph_0^{p-2} (\odot^{p-2} \na)$ 
 & $\ph_1^{p-4} (\odot^{p-4} \na) \De^1$
 & $\ph_2^{p-6} (\odot^{p-6} \na) \De^2$ 
 & $\ph_3^{p-8} (\odot^{p-8} \na) \De^3$ \\
order $p-3$: 
 & $\ph_0^{p-3} (\odot^{p-3} \na)$ 
 & $\ph_1^{p-5} (\odot^{p-5} \na) \De^1$
 & $\ph_2^{p-7} (\odot^{p-7} \na) \De^2$ 
 & $\ph_3^{p-9} (\odot^{p-9} \na) \De^3$ \\
order $p-4$: 
 & $\ph_0^{p-4} (\odot^{p-4} \na)$ 
 & $\ph_1^{p-6} (\odot^{p-6} \na) \De^1$
 & $\ph_2^{p-8} (\odot^{p-8} \na) \De^2$ 
 & $\ph_3^{p-10} (\odot^{p-10} \na) \De^3$ \\
$\vdots$ & $\vdots$ & $\vdots$ & $\vdots$ & $\vdots$
\end{tabular} \\[1ex]
Every line shows terms of the same formal order and moreover every antidiagonal
shows terms of the same level. So the ordering \nn{ord} in this case means
$$ \ph_0^{p} \rhd \ph_1^{p-2} \rhd \ph_0^{p-1} \rhd \ph_2^{p-4} \rhd 
   \ph_1^{p-3} \rhd \ph_0^{p-2} \rhd \ph_3^{p-6} \rhd \cdots $$
Observe the level $\ell(R)$ of an operator $R$ is 
increased by $k$ 
under composition with $\De^k$: 
$$ \ell(\De^k R) = \ell(R)+k. $$
Moreover, only terms of the highest level in $R$ can contribute
to terms of the highest level in $\De^kR$.

The greatest coefficient (w.r.t. $\rhd$) is $\ph_0^{p}$.  Recall $o(T)=p$
so we can assume $\ell(T) = p$ which
means $\ell(\De^4T) = p+4$. Now we consider terms of the level $p+4$
of $\De^4T$. First we commute all covariant derivatives $\nabla$ to
the right.  In fact, it is sufficient for our purpose to consider only
certain terms.  First we restrict to terms of the level $p+4$ without
a right factor $\De^4$
and then take the candidate for the greatest among these.  This is
$(\na^1\ph_0^p) (\odot^{p+1} \na) \De^3$.  Since this does not have a
right factor $\De^4$, it has to vanish since $T$ is a symmetry. Hence
$(\na^1\ph_0^p)_\boxtimes=0$, which means that $\ph_0^p$ is a
conformal Killing tensor.  Now we replace the symmetry $T$ by $T -
S_{\ph_0^p}$; this is also a symmetry.  The greatest coefficient of
$T-S_{\ph_0^p}$ is now strictly smaller (w.r.t.\ $\rhd$) than the
greatest coefficient of $T$.  (Here we have adjusted $S_{\ph_0^p}$ so the
leading term is precisely $\ph_0^{p} (\odot^p \na)$ rather than some
non-zero multiple. We will not comment further when this sort of
maneuver is used below.)  
 It is $\ph_1^{p-2}$ according to
\nn{ord}.  So now we may rename $ T-S_{\ph_0^p}$ as $T$ and continue
with the argument.

The next step is to assume $\ph_0^p=0$ and study differential
conditions imposed on $\ph_1^{p-2}$.  Here we skip this and several other steps
and we assume the greatest coefficient of $T$ is
$\ph_3^{p-6}$.  So suppose that $\ph_i^j=0$ for
$\ell(\ph_i^j) > p-3 = \ell(\ph_3^{p-6})$.  Then $\ell(T) = p-3$ and
so $\ell(\De^4T) = p+1$. We shall examine those terms of
the operator $\De^4T$ of the (highest) level $p+1$ and such that they
are without a right factor $\De^4$. To find these  it is sufficient to
consider
$$ \De^4 \left[ \ph_3^{p-6} (\odot^{p-6} \na) \De^3 + \ph_2^{p-5}
  (\odot^{p-5} \na) \De^2 + \ph_1^{p-4} (\odot^{p-4} \na) \De^1 +
  \ph_0^{p-3} (\odot^{p-3} \na) \right]. 
$$ We use the Leibniz rule to move $\De^4$ to the right in the
previous display.  We need to know the form of (level $p+1$) terms of
types $\type{p-2}{3}$, $\type{p-1}{2}$, $\type{p}{1}$ and
$\type{p+1}{0}$. The simplest case is the type $\type{p+1}{0}$, we
obtain only the term $2^4 (\na^4 \ph_0^{p-3}) \odot^{p+1} \na$. The operator 
$\odot^{p+1} \na$ does not arise in any other way, so the given term  must vanish through 
$\ph_0^{p-3}$ satisfying the obvious equation.  In the
case of the type $\type{p}{1}$ we similarly get the equation
$$ 2^4 (\na^4 \ph_1^{p-4}) (\odot^p \na)\De 
   + 2^3 \cdot 4 (\na^3 \ph_0^{p-3}) (\odot^{p} \na)\De =0.$$
Here $2^3 \cdot 4 = 2^3 \binom{4}{1} = 2^3 \binom{4}{3}$; generally
we put $C^s(4) = 2^s \binom{4}{s}$. 
The types $\type{p-2}{3}$ and
$\type{p-1}{2}$ yield two more equations which give conditions 
for the coefficients
$\ph_0^{p-3}$, $\ph_1^{p-4}$, $\ph_2^{p-5}$ and $\ph_3^{p-6}$.
Together these four equations yield the following differential equations 
for the coefficients $\ph_i^j$: \\[1ex]
\noindent
\begin{tabular}{lc@{+}c@{+}c@{+}c@{=0}}
 type $\type{p-2}{3}$: & $C^4(4) \na^4 \ph_3^{p-6}$
 & $C^3(4) \na^3 \ph_2^{p-5}$ & $C^2(4) \na^2 \ph_1^{p-4}$
 & $C^1(4) \na^1 \ph_0^{p-3}$ \\
 type $\type{p-1}{2}$: & 0 
 & $C^4(4) \na^4 \ph_2^{p-5}$ & $C^3(4) \na^3 \ph_1^{p-4}$
 & $C^2(4) \na^2 \ph_0^{p-3}$ \\
 type $\type{p}{1}$: & 0 & 0 & $C^4(4) \na^4 \ph_1^{p-4}$ 
 & $C^3(4) \na^3 \ph_0^{p-3}$ \\
 type $\type{p+1}{0}$: &0 & 0 & 0 & $C^4(4) \na^4 \ph_0^{p-3}$
\end{tabular} \\[1ex]
Here we implicitly consider the symmetric trace--free parts in every 
equation. Now applying $\na^3$ to the first equation, $\na^2$ to the second
and $\na$ to the third, and then taking the trace--free symmetric part 
in all cases,
we obtain a linear system in variables
$[\na^7 \ph_3^{p-6}]_\boxtimes$, $[\na^6 \ph_2^{p-5}]_\boxtimes$,
$[\na^5 \ph_1^{p-4}]_\boxtimes$ and $[\na^4 \ph_0^{p-3}]_\boxtimes$.
The matrix of (integer) coefficient is\\
$$ \left( \begin{array}{cccc}
   C^4(4) & C^3(4) & C^2(4) & C^1(4) \\
        0 & C^4(4) & C^3(4) & C^2(4) \\
        0 &      0 & C^4(4) & C^3(4) \\
        0 &      0 &      0 & C^4(4) 
\end{array} \right). $$
This is non-singular.
So  all the variables must vanish, and in particular
$[\na^7 \ph_3^{p-6}]_\boxtimes=0$, which is what we wanted to prove.

\vspace{1ex}

This was the case with greatest coefficient $\ph_3^{p-6}$.  It
suggests a route to solving the remaining cases, as they yield linear
systems in the same way. Actually it turns out that in each of the cases with
the greatest terms between $\ph_0^p$ and $\ph_3^{p-6}$ (which were
skipped above), the matrix of coefficients  includes 
a square ``upper right'' submatrix of the matrix above, i.e.\ a matrix
obtained by removing the first $q$ columns and the last $q$ rows for
some choice of $q$,  
that is sufficient if non-degenerate. 
That is it suffices to prove that
determinants of these matrices are nonzero. This necessitates
analysing the combinatorial coefficients $C^s(4)$ in more detail.

\vspace{1ex}

The general case is analogous; in the case of $\De^k$, $k \in \N$
we shall need the scalars
$$ C^s(k) := 2^s \binom{k}{s}, \quad 
   C^s(k):=0 \ \text{for} \ s>k $$
and matrices 
\begin{eqnarray} \label{mat} 
\begin{split}
  &\bC(k;d) \in \Mat_{k-d},\ 0 \leq d \leq k-1 \quad \text{where} \\
  &\bC(k;d)_{s,t} = C^{k-d+s-t}(k),\ 1 \leq s,t \leq k-d. 
\end{split}
\end{eqnarray}   
The matrices $\bC(k,0)$ are upper diagonal with $C^k(k)$ on
the diagonal; the matrix $\bC(4,0)$ appeared in the previous example.
In fact, $\bC(k,d)$ is obtained from $\bC(k,0)$ by removing $d$ first 
columns and $d$ last rows. Note also that considering (any) diagonal of 
$\bC(k,d)$, all the coefficients are the same. 

Clearly the $\bC(k,0)$ are regular.

\begin{theorem} \label{reg}
The matrices $\bC(k,d)$, $k \in\N$, $0 \leq d \leq k-1$ are regular.
\end{theorem}

The following proof of this Theorem  is due to 
J.\ Kadourek, of Masaryk University.
\begin{proof}

First observe that for $d=0$ the matrix $\bC$ is upper triangular with
nonzero entries on the diagonal. Thus it is regular so it is
sufficient to assume $1 \leq d \leq k-1$. Also to simplify the
notation we put $k_d := k-d$. Clearly $1 \leq k_d \leq k-1$.

It turns out to be useful 
to consider also the closely related matrix 
\begin{eqnarray} 
\begin{split}
  &\wt{\bC}(k;d) \in \Mat_{k_d},\ 0 \leq d \leq k-1 \quad \text{where} \\
  &\wt{\bC}(k;d)_{s,t} = \binom{k}{k_d+s-t},\ 1 \leq s,t \leq k_d, 
\end{split}
\end{eqnarray}   
where the latter is taken to be 0 if $s-t>d$. 
That is, the entries of $\bC$ and $\wt{\bC}$ differ by a power of $2$.
Now writing the determinant as a sum (over permutations of
$\{1,\ldots,k_d\})$) of products of entries of a matrix, one easily shows
that determinants of $\bC$ and $\wt{\bC}$ differ by a power of $2$.
That is, the matrix $\bC$ is regular if and only if $\wt{\bC}$ is regular.
We shall prove regularity for the latter. 

First recall the well-know relation
\begin{equation} \label{bineq}
  \binom{q}{m} + \binom{q}{m+1} = \binom{q+1}{m+1}, \quad q,m \geq 0.
\end{equation}
Henceforth we fix the values  $k$, $d$ from the allowed range.
The proof now consists of several series of row or column elementary
operations which change the determinant by a nonzero multiple.  During
certain stages of this process we shall obtain matrices $D_1, D_2,
D_3, D_4 \in \Mat_{k_d}$ whose determinants differ from each other
only by nonzero multiples.  The last of these, $D_4$ is upper triangular with
nonzero entries on the diagonal, and so this concludes the proof. 

The construction of $D_1$ from $\wt{\bC}$ consists of $k_d-1$ steps;
in each of these we undertake a series of elementary column
operations, as follows.  In the first step, we add the second column
to the first one, then the third column to the second and so on;
finally we add the last column to the last but one. In the second
step, we add the second column to the first one, then the third column
to the second and so on but finish by adding the $(k_d-1)$th column to
the $(k_d-2)$th column. Continuing in this way, in the last step
(i.e.\ the step number $k_d-1$) we add only the second column to the
first one. Note the determinants of $D_1$ and $\wt{\bC}$ differ by a
nonzero multiple.

Overall we obtain the matrix
\begin{equation} \label{D1}
  D_1(s,t) = \binom{k+k_d-t}{k_d+s-t}
  = \frac{(k+k_d-t)!}{(k_d+s-t)!(k-s)!};
\end{equation}
note $1 \leq k_d+s-t \leq k+k_d-t$.
The reasoning uses \nn{bineq} in every addition of two binomial numbers 
and goes as follows. Consider how the $(s,t)$-entry changes during
the procedure described in the previous paragraph.
First observe that after the $i$th step of 
elementary column operations, this entry has the form
$\binom{a_i}{k_d+s-t}$. That is, 
the ``denominator'' of the binomial number on the position $(s,t)$ does 
not change during this procedure. This follows from \nn{bineq}.
Second, the ``numerator'' of the binomial number on the $(s,t)$-position
increases by $1$ if we add the $(s,t+1)$--entry, see \nn{bineq}.
Thus the  ``numerator'' depends on the number of additions of the $(t+1)$st
column, as stated in \nn{D1}.

Now we modify the matrix $D_1$ as follows.  First we multiple the
$t$th column by $\frac{1}{(k+k_d-t)!}$, where we note that $k+k_d-t
\geq k \geq 1$. Then we multiply the $s$th row by $(k-s)!$ where $k-s
\geq 1$ because $s \leq k_d \leq k-1$.  We obtain the matrix $D_2$,
the determinants of $D_1$ and $D_2$ differ by a nonzero multiple. It
follows from the fractional form of entries of $D_1$ in \nn{D2} that
\begin{equation} \label{D2}
  D_2(s,t) = \frac{1}{(k_d+s-t)!}.
\end{equation}

We continue with the following modification of $D_2$. First we multiply
the $s$th row by $(k_d+s-1)! \geq 1$. Then we multiply the $t$th column by
$\frac{1}{(t-1)!}$, $t-1 \geq 0$ (thus $(t-1)! \geq 1$).  
The result is a matrix $D_3$, the determinants of $D_3$ and $D_2$ differ by 
nonzero multiple. It follows from \nn{D2} that 
\begin{equation} \label{D3}
  D_3(s,t) = \frac{(k_d+s-1)!}{(k_d+s-t)!(k-1)!} 
  = \binom{k_d+s-1}{k_d+s-t}.
\end{equation}

In the last stage we apply the following $k_d-1$ steps of elementary
row transformations to the matrix $D_3$. Observe that the first column
of $D_3$ has all its entries equal to $1$.  In the first step, we
subtract the $(k_d-1)$-st row from the $k_d$-th row, then we subtract
$(k_d-2)$-nd row from the $(k_d-1)$-st row and so on; finally we
subtract the first row from the second one. Thus the first column has
now $1$ as its top entry and $0$'s below this.  In the second step, we
subtract the $(k_d-1)$-st row from the $k_d$-th row, then we subtract
$(k_d-2)$-nd row from the $(k_d-1)$-st row and so on, as before except in
this step we finish at the point of subtracting the 2nd row from the
3rd row. Continuing in this way, in the last step we subtract only
$(k_d-1)$-st row from the $k_d$-th row.  We shall denote the resulting
matrix by $D_4$.

It turns out $D_4$ is upper triangular with all entries on the
diagonal equal to $1$. To show this note we use \nn{bineq} at every
step of the above procedure. In fact, the final form of $D_4$ can be
foreseen already from the first step, after which we obtain a matrix
that we shall denote $O \in \Mat_{k_d}$.  We already know the first column
of $O$ is $(1,0,\ldots,0)^T$. From this it follows that in the
second step we effectively work only with submatrix of $O$ with
entries $(s,t)$, $2 \leq s,t \leq k_d$.  Since
$$ O(s,t) = \binom{k_d+s-2}{k_d+s-t} 
   = D_3(s-1,t-1),
   \quad 2 \leq s,t \leq k_d $$
using \nn{bineq}, we see this submatrix of $O$ is exactly the submatrix
of $D_3$ without the last row and the last column. Applying the second
step to the displayed submatrix corresponds to applying the first step
to the corresponding submatrix of $D_3$ (the last row and column  
clearly have no influence on the previous ones). These observations yield 
an inductive procedure which demonstrates the claimed form of $D_4$.
\end{proof}

\begin{proposition} \label{biggest}
Let $(T,T')$ be a normal symmetry of $\De^k$ and suppose that, in a
standard expression for $T$, $\ph_r^p (\odot^p \na) \De^r$  is
the greatest non-zero term of $T$ with respect to $\rhd$. Then
$[\na^{2r+1} \ph_r^p]|_\boxtimes =0$.
\end{proposition}

\begin{proof}
The ordering $\lhd$ can be equivalently described as
$\ph_i^j \lhd \ph_{i'}^{j'}$ if and only if either
$i+j < i'+j'$ or $i+j = i'+j'$ and $i<i'$. Thus
$$ \De^kT=T'\De^k, \quad 
   T  = \ph_r^p (\odot^p \na) \De^r+ \sum_{
   \substack{i<k\\
i+j < r+p\ \text{or} \\ (i+j = r+p)~ \wedge ~(i<r)}} 
   \ph_i^j (\odot^j \na) \De^i. $$
Note $\ph_r^p$ might not be a leading term of $T$.

Note, $\ell(T) = p+r$ and $\ell(\De^kT) = p+r+k$. 
We shall discuss the terms of the highest level in $\De^k T$. For this
it is sufficient to apply $\De^k$ only to level $p+r$ terms of $T$.
That is, we need to understand the right hand side of 
\begin{eqnarray*} 
  &\De^k \left[ \ph_r^p (\odot^p \na) \De^r +
   \ph_{r-1}^{p+1} (\odot^{p+1} \na) \De^{r-1} + \ldots +
   \ph_0^{p+r} (\odot^{p+r} \na) \right]  - F\De^k \\
  &=\ps_{k-1}^{p+r+1} (\odot^{p+r+1} \na) \De^{k-1} +
   \ps_{k-2}^{p+r+2} (\odot^{p+r+2} \na) \De^{k-2} + \ldots +
   \ps_0^{p+r+k} (\odot^{p+r+k} \na) + \text{llt}
\end{eqnarray*}
where $F$ is a differential operator.  Here ``llt''
denotes terms of the level at most $p+r+k-1$ (with powers of $\Delta$
strictly less than $k$) and $\ps_i^j$ is of type $\type{j}{i}$. Since $i<k$
for every $\ps_i^j$ on the right had side, imposing the symmetry
condition, each of these terms has to vanish.  This yields $k$
differential conditions
$$ \ps_{k-1}^{p+r+1} (\odot^{p+r+1} \na) \De^{k-1} =0,
   \ps_{k-2}^{p+r+2} (\odot^{p+r+2} \na) \De^{k-2} =0, \ldots,
   \ps_0^{p+r+k} (\odot^{p+r+k} \na) =0. $$
Thus $\ps_{k-q-1}^{p+r+q+1}=0$ for $q \in \{0,\ldots,k-1\}$.  For our
purposes it turns out to be sufficient to take $q$ in the (in general
smaller) range $ \{0,\ldots,r\}$. So we have $r+1$ differential
conditions. Now fix such a $q$; we have more explicitly
$$ \ps_{k-q-1}^{p+r+q+1} = \bigl[ a_{q,0} \na^{r+q+1} \ph_r^p +
  a_{q,1} \na^{r+q} \ph_{r-1}^{p+1} + \ldots + a_{q,r} \na^{q+1}
  \ph_0^{p+r} \bigr]|_\boxtimes $$ for some integer coefficients
$a_{q,q'}$, $q' \in \{0,\ldots,r\}$.  Via the Leibniz rule and a
counting argument, it is straightforward to verify that $a_{q,q'} =
C^{r+q-q'+1}(k)$.  
Recall $\ps_{k-q-1}^{p+r+q+1}=0$ hence the right hand side of the
previous display vanishes. Finally, let us apply $\na^{r-q}$ to both
sides of the previous display. Projecting to the Cartan component, we
obtain
$$ \bigl[ 
   C^{r+q+1}(k) (\na^{2r+1} \ph_r^p) 
   + C^{r+q}(k) (\na^{2r} \ph_{r-1}^{p+1}) + \ldots 
   + C^{q+1}(k) (\na^{r+1} \ph_0^{p+r}) \bigr]|_\boxtimes =0.$$
This is a linear equation in the $r+1$ variables 
$(\na^{2r+1} \ph_r^p)|_\boxtimes$,
$(\na^{2r} \ph_{r-1}^{p+1})|_\boxtimes$, $\ldots$, 
$(\na^{r+1} \ph_0^{p+r})|_\boxtimes$.
These variables obviously do not depend on $q$. That is 
for every $q \in \{0,\ldots,r\}$ we obtain one equation in these 
variables. Overall we have a system of $r+1$ linear equations in $r+1$ 
variables
$(\na^{2r+1} \ph_r^p)|_\boxtimes$,
$(\na^{2r} \ph_{r-1}^{p+1})|_\boxtimes$, $\ldots$, 
$(\na^{r+1} \ph_0^{p+r})|_\boxtimes$. The integer coefficients are
$a_{q,q'} = C^{r+q-q'+1}(k) = C^{(r+1)+(q+1)-(q'+1)}(k)$,
$q,q' \in \{0,\ldots,r\}$ thus the $(r+1) \times (r+1)$ matrix of integer 
coefficients is exactly $\bC(k,d)$ for $d = k-r-1$ from \nn{mat}. 
(Note $r<k$ hence $d \in \{0,\ldots,k-1\}$.) 
But matrices $\bC(k,d)$ are regular according to Theorem \ref{reg}. 
Therefore this linear system has only the zero solution, i.e.\
$$ (\na^{2r+1} \ph_r^p=0)|_\boxtimes=0, \na^{2r}
(\ph_{r-1}^{p+1})|_\boxtimes=0, \ldots, (\na^{r+1}
\ph_0^{p+r})|_\boxtimes=0. $$ In particular $(\na^{2r+1}
\ph_r^p)|_\boxtimes=0$, which is what we wanted to prove.
\end{proof}

Finally we have the key theorem of this section. By an obvious
induction this establishes the second part of Theorem
\ref{main}. 
\begin{theorem} \label{class}
Let $(S,S')$ be a normal symmetry of $\De^k$ and suppose that, in a
standard expression for $S$, $\ph_r^p (\odot^p \na) \De^r$, $r<k$ is a
leading term. Then $[\na^{2r+1} \ph_r^p]|_\boxtimes =0$.
\end{theorem}

\noindent This establishes the second part of Theorem \ref{main}.
Note that using the conformal metric, we can view all $p+2r+1$ abstract
indices of $\na^{2r+1} \ph_r^p$ as contravariant. Then the projection
to the Cartan component in $[\na^{2r+1} \ph_r^p]|_\boxtimes =0$ simply
means taking the symmetric trace-free part.

\begin{proof}
Consider the coefficients of the maximal level $\ell(S)$ of 
$S$; among them, denote by $\ps_i^j$ the term of the highest order.
In the other words, $\ps_i^j$ is the greatest coefficient in $S$ w.r.t.\ $\lhd$.
Now $[\na^{2i+1} \ps_i^j]|_\boxtimes =0$ according to Proposition
\ref{biggest} hence $\ps_i^j$ yields the corresponding canonical symmetry 
$(S_\ps^{}, S'_{\ps})$ of $\De^k$. Therefore 
$(S-S_\ps^{},S'-S'_{\ps})$ is also a symmetry of $\De^k$. 

First observe using Proposition \ref{can:prop} (iii) that 
the leading terms of $S$ and $S-S_\ps$ can differ only if 
$\ps_i^j (\odot^j \na) \De^i$ is a leading term of $S$. But in that
case we have proved the theorem for $\ps_i^j (\odot^j \na) \De^i$. 
Therefore, it is sufficient to prove the theorem for
$S-S_\ps$. So we can take $S:=S-S_\ps$ and continue inductively.

Proposition \ref{can:prop} (ii) guarantees that
the greatest term of $S:=S-S_\ps$ is smaller than the greatest term
of $S$. Hence this induction w.r.t.\ $\lhd$ is finite.
\end{proof}

\section{Algebra of symmetries}

Here we shall prove Theorem \ref{algmain}. Recall that the finite
dimensional space of solutions of (\ref{1BGG}) 
may be realised as a standard
linear ``matrix'' representation of $\frak{g} = \frak{so}_{s+1,s'+1}$
via the map from solutions to parallel tractors $\ph \mapsto I_\ph$.
In the case of conformal Killing vectors (i.e.\ (\ref{1BGG}) with
$p=1$, $r=0$) the range space is $\frak{g}$, on which $\frak{g}$ acts by the
adjoint representation.  Then the identification of $\frak{g}$ with
differential symmetries is given by the mapping 
$\frak{g} \ni I_\ph
\mapsto S_\ph = I_\ph^\form{A} \bbD_\form{A}$, as a special case of
\nn{can}. 
The mapping $S_\ph = I_\ph^\form{A} \bbD_\form{A}$  extends to
\begin{equation} \label{key}
\frak{g} \otimes \frak{g} \otimes \cdots \frak{g} \ni
I_{\ph_1} \otimes \cdots \otimes I_{\ph_m} \mapsto
S_{\ph_1} \cdots S_{\ph_m}, \quad m \geq 1 ~,
\end{equation}
and hence to the full tensor algebra $\bigotimes \frak{g}$ by linearity.

The first step in the proof of Theorem \ref{algmain} is to express the
composition $S_{\ph} S_{\bar{\ph}}$ for $I_{\ph},I_{\bar{\ph}} \in
\frak{g}$ in terms of canonical symmetries. This is done \cite[Theorem
  5.1]{EaLeLap2} and necessarily our results must agree with those from
their construction (as uniqueness of the low order symmetries involved
is easily verified).  We present the details here to keep this text
self--contained and also because we derive the formulae for all
conformally flat manifolds.

Putting $I:= I_\ph, \bar{I} := I_{\bar{\ph}}$ to simplify the notation,
one has 
\begin{equation} \label{SS}
S_\ph S_{\bar{\ph}} = 
I^\form{A} \bbD_\form{A} \bar{I}^\form{B} \bbD_\form{B} =
I^\form{A} \bar{I}^\form{B} \bbD_\form{A} \bbD_\form{B}.
\end{equation}
on $\cE[w]$, since $I$ is parallel.
This gives an explicit and key link between the algebraic structure of
symmetries $\cA_k$ and operations on the tensor algebra $\bigotimes
\frak{g}$.  We shall consider the displayed operator acting on
$\cE[w]$ for all $w \in \R$ at this stage.

We need to decompose 
$\bbD_\form{A} \bbD_\form{B}$ into irreducible components. 
Using the definition of $\bbD_\form{A}$, a direct computation shows that
\begin{eqnarray} \label{DD}
\begin{split} 
\bbD_\form{A} \bbD_\form{B} f =
& 4w^2\W_\form{A}\W_\form{B} f 
- 4w\X_\form{A}^{\,a}\Y_\form{B}^{\,b} \bg_{ab}f \\
& +4(w-1)\X_\form{A}^{\,a}\W_\form{B}^{} \na_a f 
+ 4w \W_\form{A}^{}\X_\form{B}^{\,b} \na_b f 
+ 4 \X_\form{A}^{\,a} \Z_\form{B}^{\,\form{b}} \bg_{ab^0} \na_{b^1} f \\ 
& + 4\X_\form{A}^{\,a} \X_\form{B}^{\,b} (\na_a \na_b + wP_{ab}) f.
\end{split}
\end{eqnarray}
From this one easily verifies that
\begin{eqnarray} \label{DDdec}
\begin{split}
\frac{1}{2} (\bbD_\form{A} \bbD_\form{B} 
+ \bbD_\form{B} \bbD_\form{A}) &=&&
\frac{1}{2} (\bbD_\form{A} \bbD_\form{B} 
+ \bbD_\form{B} \bbD_\form{A})|_\boxtimes
+\frac{4}{n}h_{A^0B^0} \bbD^2_{(A^1B^1)_0} \\
&&&+\frac{2}{(n+1)(n+2)} h_{A^0B^0}h_{A^1B^1} 
\bbD^\form{A} \bbD_\form{A}, \\
\frac{1}{2} (\bbD_\form{A} \bbD_\form{B} 
- \bbD_\form{B} \bbD_\form{A}) &=&& 
3 h_{A^0[A^1}\bbD_\form{B]} = -2 h_{A^0B^0} \bbD_{A^1B^1}.
\end{split}
\end{eqnarray}
Hence we need the irreducible components $\myng{2,2}_0$, $\myng{2}_0$,
$\myng{1,1}$ and $\R$ of
$I^\form{A}\bar{I}^\form{B}$, cf.\ \nn{gg}.
Explicitly, we put
\begin{eqnarray} \label{IIirr}
\begin{split}
& \langle I,\bar{I} \rangle := -4n I^\form{A} \bar{I}_\form{A} \in \R, \\
& [I,\bar{I}]^\form{A} := 4I^{A^0P} \bar{I}_p{}^{A^1} \in \myng{1,1}, \\
& (I \bu \bar{I})^{BB'} := \frac{4}{n} I^{P(B} \bar{I}_P{}^{B')_0}
\in \myng{2}_0
\end{split}
\end{eqnarray}
and we denote by $(I \boxtimes \bar{I})^\form{AB}$ the trace--free
part of the Young projection $\myng{2,2}$ applied to $I^\form{A}
\bar{I}^\form{B}$. Using this notation, the projection and
decomposition of $I^\form{A} \otimes \bar{I}^\form{B}$ into
its irreducible components in $\myng{2,2}$, $\bbR$, $\myng{1,1}$ and
$\myng{2}_0$ is given by
\begin{eqnarray} \label{IIdecomp}
\begin{split}
\myng{1,1} \otimes \myng{1,1} \ni I^\form{A} \otimes \bar{I}^\form{B} \mapsto
&(I \boxtimes \bar{I})^\form{AB} 
- \frac{1}{2n(n+1)(n+2)} h^{A^0B^0}h^{A^1B^1} \langle I,\bar{I} \rangle \\
& + \frac{1}{n} h^{A^0B^0} [I,\bar{I}]^{A^1B^1}
+ h^{A^0B^0} (I \bullet \bar{I})^{A^1B^1}.
\end{split}
\end{eqnarray}

Using the computation above, we easily recover \cite[Theorem 5.1]{EaLeLap2}:

\begin{theorem} \label{dec2can}
Let $\ph^a, \bar{\ph}^a \in \cE^a$ be conformal Killing fields corresponding
to $I^\form{A} := I_\ph^\form{A}$ and 
$\bar{I}^\form{A} := I_{\bar{\ph}}^\form{A}$ in 
$\frak{g} = \frak{so}_{s+1,s'+1}$.
Then
$$
S_{\ph} S_{\bar{\ph}} f = 
(I \boxtimes \bar{I})^\form{AB} \bbD_\form{A}\bbD_\form{B} f
+ (I \bu \bar{I})^{BB'} \bbD^2_{BB'}f 
+ \frac{1}{2} [I,\bar{I}]^\form{A} \bbD_\form{A} f
+ \frac{w(n+w)}{n(n+1)(n+2)} \langle I,\bar{I} \rangle f
$$
for $f \in \cE[w]$, cf. \nn{ggcomp}. 
The four summands on the right hand side are canonical 
symmetries, explicitly
\begin{itemize}
\item $(I \boxtimes \bar{I})^\form{AB} \bbD_\form{A}\bbD_\form{B} = S_{\Ph}$
for $\cE^{(ab)_0} \ni \Ph^{ab} = \ph^{(a} \bar{\ph}^{b)_0}$,
\item $(I \bu \bar{I})^{BB'} \bbD^2_{BB'} = S_\Ph$ for
$\cE[2] \ni \Ph = \frac{1}{n} \ph^a \bar{\ph_a}$,
\item $[I,\bar{I}]^\form{A} \bbD_\form{A} = S_\Ph$ for
$\cE^a \ni \Ph^a 
= \ph^b\na_b \bar{\ph}^a - \bar{\ph}^b\na_b \ph^a$
(the Lie bracket of vector fields),
\item $\R \ni \langle I,\bar{I} \rangle = -4nI^\form{A} \bar{I}_\form{A}
= -2[\ph^a\na_a \na_b\bar{\ph}^b + \bar{\ph}^a\na_a\na_b\ph^b] 
+n(\na_a \ph^b)(\na_b \bar{\ph}^a)
-\frac{n-2}{n} (\na_a \ph^a)(\na_b \bar{\ph}^b)
-4nP_{ab} \ph^a\bar{\ph}^b.$  
\end{itemize}
In all these cases, the section $\Ph$ is a solution of the corresponding
equation \nn{1BGG}.  
\end{theorem}

\begin{proof}
The statement puts together the previous computations.
Following \nn{SS}, we need to decompose 
$I^\form{A} \bar{I}^\form{B} \bbD_\form{A} \bbD_\form{B}$
into canonical symmetries. This is provided by contracting right hand sides 
of \nn{IIdecomp} and \nn{DDdec}.
Using in addition $\bbD^\form{A} \bbD_\form{A} f = -2w(n+w)f$ for 
$f \in \cE[w]$ 
(which easily follows from \nn{DD}), the right hand side of
$S_\ph S_{\bar{\ph}}$ in the display above follows.

The components $I \boxtimes \bar{I}$, $I \bu \bar{I}$, $[I,\bar{I}]$
and $\langle I,\bar{I} \rangle$ are parallel (and irreducible) thus
their projecting parts $\Ph$ are solutions of the corresponding
equation from the family \nn{1BGG}. To prove the theorem, it remains
to identify how are these solutions are built from $\ph^a, \bar{\ph}^a
\in \cE^a$. Note
\begin{equation} \label{Iph}
I^\form{A} = \Y^\form{A}_{\,a} \ph^a 
+\frac{1}{2} \Z^\form{A}_{\,\form{a}} \na^{a^0} \ph^{a^1} 
+ \frac{1}{n} \W^\form{A} \na_a \ph^a 
+ \X^\form{A}_{\,a}[\frac{1}{n} \na_a \na_b\ph^b + P_{ab}\ph^b] 
\end{equation}
and similarly for $\bar{I}^\form{A}$ \cite{GoSiCkf}.
Now  the explicit form of such $\Ph$ for irreducible components of 
$I^\form{A} \otimes \bar{I}^\form{B}$ is easily obtained from
\nn{IIirr} for $I \bu \bar{I}$, $[I,\bar{I}]$ and $\langle I,\bar{I} \rangle$.
Since $\frac{1}{2}(I^\form{A} \bar{I}^\form{B} + I^\form{B} \bar{I}^\form{A})$
has the projecting part $\ph^{(a} \bar{\ph}^{b)}$, the case 
$I \boxtimes \bar{I}$ follows by irreducibility.
\end{proof}

To finish the proof of Theorem \ref{algmain}, observe the following.
First we have an associative algebra morphism 
$$
\bigotimes \frak{g} \to \mathcal{A}_k
$$ determined by \nn{key}. That this is surjective is an easy
consequence of Theorem \ref{csymc} since the canonical symmetries
$S_\phi$ of \nn{can} clearly arise in the range of \nn{key}.  We want
to find all corresponding relations, that is identify the two sided
ideal annihilated by this map.  The ideal certainly contains \nn{gen},
as follows from Theorem \ref{dec2can} with $w = -\frac{n}{2}+k$. That
it also contains $\boxtimes^{2k} \myng{1}$ is due to the following
result.
\begin{lemma}\label{extra}
Assume $I \in \boxtimes^{2k} \myng{1}$ is parallel. Then 
$I = I_\ph$ for $\ph \in \cE[2k]$ and 
$S_\ph =\ph  P_k:\cE[-\frac{n}{2}+k] \to \cE[-\frac{n}{2}+k]$.
\end{lemma}
\

\begin{proof}
$I \in \boxtimes^{2k} \myng{1}$ means 
$I^{A_1A_1' \cdots A_kA_k'} \in \cE^{(A_1A_1' \cdots A_kA_k')_0}$ and
$I = I_\ph$ for $\ph \in \cE[2k]$ is due to the irreducibility of $I$ and the fact that is parallel.
Then 
$$
S_\ph = I^{A_1^{}A_1' \cdots A_k^{}A_k'}_\ph 
\cD^2_{A_1^{}A_1'} \cdots \cD^2_{A_k^{}A_k'}.
$$
Now observe 
$\cD^2_{(CD)_0} = -X_{(C}D_{D)_0}$ and $X_{(C}D_{D)_0} = D_{(C}X_{D)_0}$,
cf.\ \nn{[D,X]}. On the other hand $D_{(A_1} \cdots D_{A_k)_0} = 
(-1)^k X_{(A_1} \cdots X_{A_k)_0} P_k$ on $\cE[-\frac{n}{2}+k]$ \cite{Gosrni,GoPetLap}.
Thus 
$\cD^2_{A_1^{}A_1'} \cdots \cD^2_{A_k^{}A_k'} = 
X_{A_1^{}}X_{A_1'} \cdots X_{A_k^{}}X_{A_k'}P_k^{}$ on $\cE[-\frac{n}{2}+k]$.
The rest follows from the relation between $\ph$ and $I_\ph$ in \nn{split}.
\end{proof}

We have found the generators of the ideal in $\bigotimes \frak{g}$
described in Theorem \ref{algmain}; it remains to show that this ideal
large enough to have $\cA_k$ as the resulting quotient.  Essentially we follow
\cite{EaLap,EaLeLap2} where cases $k=1$ and $k=2$ are studied. We
assume $k \geq 1$ here.  Since we know $\cA_k$, as a vector space,
from \nn{Avec}, it is sufficient to consider the corresponding graded
algebra (i.e.\ the symbol algebra of $\cA_k$.)  The corresponding
graded ideal contains $I_1 \otimes I_2 - I_1 \boxtimes I_2 - I_1 \bu
I_2$ for $I_1, I_2 \in \frak{g}$, cf.\ \nn{gen}, hence it contains
$\frak{g} \wedge \frak{g}$. 
Therefore we can pass to $\bigodot
\frak{g}$ 
and we write $\cI$ for the ideal in $\bigodot
\frak{g}$ which is the image of the ideal of Theorem \ref{algmain}.
We claim  
that as a graded structure $\cA_k = \bigoplus \cA_{k,t}$
where the $\cA_{k,t}$ are defined as the submodules
satisfying
$$
\cA_{k,t} = \Bigl\{
X \in \raisebox{-13pt}{$
\overbrace{\begin{picture}(60,30)
\put(0,5){\line(1,0){60}}
\put(0,15){\line(1,0){60}}
\put(0,25){\line(1,0){60}}
\put(0,5){\line(0,1){20}}
\put(10,5){\line(0,1){20}}
\put(20,5){\line(0,1){20}}
\put(50,5){\line(0,1){20}}
\put(60,5){\line(0,1){20}}
\put(35,10){\makebox(0,0){$\cdots$}}
\put(35,20){\makebox(0,0){$\cdots$}}
\end{picture}}^{t}$} \quad \text{\ s.t.\ }
\underbrace{\trace( \ldots (\trace}_k (X)..)) =0
\Bigr\} \subseteq \bigodot {\!}^t \frak{g}.
$$ The traces are taken via the tractor metric and note that the trace
condition arises from Lemma \ref{extra} above. As a vector space this
is the right answer as, by standard representation theory, $\cA_{k,t}
= \bigoplus_{j+2i = t} \cK_i^j$, $t \geq 1$.  
To finish the
proof, we need to show $\bigodot^t \frak{g} = \cA_{k,t} \oplus \cI_t$
(as vector spaces) where $\cI_t = \cI \cap \bigodot^t \frak{g}$, $t
\geq 1$.  This is based on the following

\begin{lemma*}
Assume $t \geq 3$, $k \geq 1$. Then
$$
\Bigl(\: \myng{1,1} \otimes \cA_{k,t-1} \Bigr) 
\cap \Bigl( \cA_{k,t-1} \otimes \myng{1,1} \:\Bigr) = 
\begin{cases}
\cA_{k,t} & t\not=2k \\
\cA_{k,t} \oplus \raisebox{-1pt}{$\boxtimes^{2k}$} \myng{1}
& t=2k.
\end{cases}
$$
\end{lemma*}

\begin{proof}
The case $t<2k$ follows from \cite[Theorem 2]{EaCa} or can be easily
checked directly. Assume $t>2k$.  The inclusion ``$\supseteq$'' is
obvious. To show ``$\subseteq$'' consider the tensor $F^\form{A_1
\ldots A_t}$ in the left hand side of the display.  Then
$$
F^\form{A_1 \ldots A_i  \ldots A_j \ldots A_t} 
= F^\form{A_1 \ldots A_j \ldots A_i \ldots  A_t}
$$
for any $1 \leq i<j \leq t$. From this it easily follows that the skew 
symmetrization over any three indices of $F$ is zero. 
(This and the last display also follow from \cite[Theorem 2]{EaCa}.)
Now any
composition of $k$ traces applied to $F$ affects $2k$ indices among
$2t$ indices $A_1^0,A_1^1, \ldots, A_t^0,A_t^1$, i.e.\ at most $2k$
form indices among $\form{A}_1, \ldots, \form{A}_t$.  Thus there is a
free form index $\form{A}_i$ (as $t>2k$) and the inclusion
``$\subseteq$'' follows from the symmetry 
given by the previous display.

Assume $t=2k$. Following the previous case ``$\subseteq$'', the difference 
appears
only if a composition of $k$ traces affects all $2k$ form indices
of $F$. After taking of such composition of traces we obtain a tensor in
$\bigodot^t \myng{1}$ and one easily sees this tensor is trace free. 
On the other hand, for any symmetric trace free tensor 
$G^{A_1^0 \ldots A_{2k}^0} \in  \raisebox{-1pt}{$\boxtimes^{2k}$} \myng{1}$ 
one has
\begin{equation} \label{incl}
G^{A_1^0 \ldots A_{2k}^0} h^{A_1^1 \cdots A_{2k}^1}  
\in  \Bigl(\: \myng{1,1} \otimes \cA_{k,t-1} \Bigr) 
\cap \Bigl( \cA_{k,t-1} \otimes \myng{1,1} \:\Bigr)
\end{equation}
which can be easily verified by direct computation. Here
$h^{A_1^1 \cdots A_{2k}^1} = h^{(A_1^1A_2^1} \cdots h^{A_{2k-1}^1A_{2k}^1)}$
and recall we implicitly skew over the couples $A_i^0A_i^1$ for
$1 \leq i \leq 2k$.
\end{proof}

The final step is to use 
that for each $s$, there is (by standard theory) a projection $
\odot^s\frak{g} \to \cA_{k,s}$ and that the induced projections $P_t:
\bigodot^t \frak{g} \to \frak{g} \otimes \cA_{k,t-1}$ and $Q_t:
\bigodot^t \frak{g} \to \cA_{k,t-1} \otimes \frak{g}$ 
have kernel in, respectively $\frak{g} \otimes \cI_{t-1}$ and
$\cI_{t-1} \otimes \frak{g}$ (and hence in both cases in $\cI_t$)
where for each non-negative integer $s$, $\cI_{s} = \cI \cap
\bigodot^{s} \frak{g}$.  Therefore, by obvious dimensional considerations, 
\begin{equation} \label{bang}
\bigodot\!{}^t \, \myng{1,1} = 
(\im P_t \cap \im Q_t) \oplus (\ker P_t + \ker Q_t), \ \
t \geq 3
\end{equation}
and the claim above and then Theorem \ref{algmain} follow by induction.

\end{document}